\documentclass[runningheads, envcountsame]{llncs}
\usepackage[T1]{fontenc}
\usepackage{comment}

\usepackage{amssymb}
\usepackage{graphicx}
\usepackage{xcolor}
\usepackage[linesnumbered,  ruled,  vlined	]{algorithm2e}
\SetKwInOut{Input}{Input}
\SetKwInOut{Output}{Output}
\usepackage{subcaption}
\usepackage{tikz}
\usetikzlibrary{arrows,automata}
\usetikzlibrary{arrows.meta}
\usetikzlibrary{decorations.pathreplacing, calligraphy}
\usetikzlibrary{decorations.pathmorphing}
\usetikzlibrary{shapes.geometric}
 \tikzstyle{br} = [decorate, ultra thick, decoration = {calligraphic brace}]

\newcommand{\dcut}{{\sc $d$-Cut}}
\newcommand{\mc}{{\sc Matching Cut}}

\newcommand{\NP}{{\sf NP}}
\newcommand{\dist}{{\sf dist}}
\newcommand{\radius}{{\sf radius}}
\newcommand{\diam}{{\sf diameter}}

\newcommand{\ssi}{\subseteq_i}
\newcommand{\si}{\supseteq_i}

\newtheorem{observation}[theorem]{Observation}
\newtheorem{nclaim}{Claim}
\counterwithin{nclaim}{theorem}

\newenvironment{claimproof}[1]{\par\noindent{\textit{Proof of the Claim.}}
\enspace#1}{\hfill \textcolor{gray}{$\vartriangleleft$}}

\definecolor{nicered}{RGB}{204,0,0}
\definecolor{lightblue}{RGB}{153,204,255}
\tikzstyle{vertex}=[thin,circle,inner sep=0.cm, minimum size=1.7mm, fill=black, draw=black]
 \tikzstyle{svertex}=[thin,circle,inner sep=0.cm, minimum size=1.3mm, fill=black, draw=black]
 \tikzstyle{bvertex}=[thin,circle,inner sep=0.cm, minimum size=1.7mm, fill=lightblue, draw=lightblue]
 \tikzstyle{rvertex}=[thin,circle,inner sep=0.cm, minimum size=1.7mm, fill=nicered,draw=nicered]
 \tikzstyle{evertex}=[thin,circle,inner sep=0.cm, minimum size=1.7mm, fill=none,draw=black]
 \tikzstyle{edge}=[thick, draw = gray]
 \tikzstyle{tedge}=[ultra thick, draw = black]
 \tikzstyle{tredge}=[ultra thick, draw = nicered]
 \tikzstyle{tbedge}=[ultra thick, draw=lightblue]
 \tikzstyle{redge}=[thick, draw = nicered]
 \tikzstyle{bedge}=[thick, draw = lightblue] 
 \tikzstyle{gedge}=[thick, draw = nicegreen] 
 \tikzstyle{brace} = [decorate, ultra thick, decoration = {calligraphic brace}]

\begin{document}

\title{Finding $d$-Cuts in Graphs of Bounded Diameter, Graphs of Bounded Radius and $H$-Free Graphs\thanks{An extended abstract of this paper has appeared in the proceedings of WG 2024~\cite{LMPS24}.}}

\titlerunning{Finding $d$-Cuts in Graphs}

\author{Felicia Lucke\inst{1}\orcidID{0000-0002-9860-2928} \and Ali Momeni\inst{2}\orcidID{0009-0009-8280-7847}\and Dani\"el Paulusma\inst{3}\orcidID{0000-0001-5945-9287}  \and Siani Smith\inst{4}\orcidID{0000-0003-0797-0512}}

\authorrunning{F. Lucke, A. Momeni, D. Paulusma, S. Smith}

\institute{LIP, ENS Lyon, Lyon, France, \email{felicia.lucke@ens-lyon.fr}
\and
Faculty of Computer Science, UniVie Doctoral School Computer Science DoCS, University of Vienna, Austria\\ 
\email{ali.momeni@univie.ac.at}
\and
Department of Computer Science, Durham University, Durham, UK \email{daniel.paulusma@durham.ac.uk}
\and  
Department of Computer Science, Loughborough University, Loughborough, UK \email{s.smith16@lboro.ac.uk}}

\maketitle             

\begin{abstract}
The \dcut{} problem is to decide whether a graph has an edge cut such that each vertex has at most $d$ neighbours on the opposite side of the cut. If $d=1$, we obtain the intensively studied {\sc Matching Cut} problem. The {\sc $d$-Cut} problem has been studied as well, but a systematic study for special graph classes was lacking. We initiate such a study and consider classes of bounded diameter, bounded radius and $H$-free graphs. We prove that for all $d\geq 2$, \dcut{} is polynomial-time solvable for graphs of diameter~$2$, $(P_3+P_4)$-free graphs and $P_5$-free graphs. These results extend known results for $d=1$. However, we also prove several \NP-hardness results for \dcut{} that contrast known polynomial-time results for $d=1$. Our results 
lead to full dichotomies for bounded diameter and bounded radius and to 
almost-complete dichotomies for $H$-free graphs.
\keywords{matching cut  \and $d$-cut \and diameter \and radius \and $H$-free graph}
\end{abstract}

\section{Introduction}\label{s-intro}

We consider the generalization \dcut{} of a classic graph problem {\sc Matching Cut} ($1$-{\sc Cut}). First, we explain the original graph problem. Consider a connected graph $G=(V,E)$, and let $M\subseteq E$ be a subset of edges of $G$. The set $M$ is an {\it edge cut} of $G$ if it is possible to partition $V$ into two non-empty sets $B$ (set of {\it blue} vertices) and $R$ (set of {\it red} vertices) in such a way that $M$ is the set of all edges with one end-vertex in $B$ and the other one in~$R$.
 Now, suppose that $M$ is in addition also a {\it matching}, that is, no two edges in $M$ have a common end-vertex. Then $M$ is said to be a {\it matching cut}. See Figure~\ref{f-examplesdcut} for an example.

\begin{figure}[t]
\centering
\scalebox{1}{
\begin{tikzpicture}

\begin{scope}[scale = 0.8]
\node[bvertex] (v1) at (0,2){};
\node[bvertex] (v2) at (0,1){};
\node[bvertex] (v3) at (1,2){};
\node[bvertex] (v4) at (1,1){};
\node[bvertex] (v5) at (1,0){};
\node[rvertex] (u1) at (2,2){};
\node[rvertex] (u2) at (2,1){};
\node[rvertex] (u3) at (2,0){};
\node[rvertex] (u4) at (3,2){};
\node[rvertex] (u5) at (3,1){};
\node[rvertex] (u6) at (3,0){};

\draw[edge](v1) -- (v2);
\draw[edge](v1) -- (v3);
\draw[edge](v2) -- (v3);
\draw[edge](v2) -- (v4);
\draw[edge](v4) -- (v5);

\draw[tedge](v4) -- (u2);
\draw[tedge](v5) -- (u3);

\draw[edge](u1) -- (u2);
\draw[edge](u1) -- (u4);
\draw[edge](u1) -- (u5);
\draw[edge](u1) -- (u6);
\draw[edge](u2) -- (u4);
\draw[edge](u3) -- (u5);
\draw[edge](u3) -- (u6);

\end{scope}

\begin{scope}[shift = {(4,0)}, scale = 0.8]
\node[bvertex] (v1) at (0,2){};
\node[bvertex] (v2) at (0,1){};
\node[bvertex] (v3) at (1,2){};
\node[bvertex] (v4) at (1,1){};
\node[bvertex] (v5) at (1,0){};
\node[rvertex] (u1) at (2,2){};
\node[rvertex] (u2) at (2,1){};
\node[rvertex] (u3) at (2,0){};
\node[rvertex] (u4) at (3,2){};
\node[rvertex] (u5) at (3,1){};
\node[rvertex] (u6) at (3,0){};

\draw[edge] (v1) -- (v2);
\draw[edge] (v1) -- (v3);
\draw[edge] (v1) -- (v4);
\draw[edge] (v1) -- (v5);
\draw[edge] (v2) -- (v3);
\draw[edge] (v2) -- (v4);
\draw[edge] (v2) -- (v5);
\draw[edge] (v3) -- (v4);
\draw[edge] (v3) to [bend right = 25] (v5);
\draw[edge] (v4) -- (v5);

\draw[tedge] (v3) -- (u1);
\draw[tedge] (v4) -- (u1);
\draw[tedge] (v4) -- (u2);
\draw[tedge] (v4) -- (u3);

\draw[edge] (u1) -- (u2);
\draw[edge] (u1) -- (u4);
\draw[edge] (u1) -- (u6);
\draw[edge] (u2) -- (u3);
\draw[edge] (u2) -- (u4);
\draw[edge] (u2) -- (u5);
\draw[edge] (u2) -- (u6);
\draw[edge] (u3) -- (u4);
\draw[edge] (u3) -- (u5);
\draw[edge] (u3) -- (u6);
\draw[edge] (u4) -- (u5);
\draw[edge] (u5) -- (u6);

\end{scope}

\begin{scope}[shift = {(8,0)},scale=0.8]
	\node[vertex] (v1) at (0,2){};
	\node[vertex, label=left:$u$] (v2) at (0,1){};
	\node[vertex] (v3) at (0,0){};
	\node[vertex] (v4) at (3,2){};
	\node[vertex, label=right:$v$] (v5) at (3,1){};
	\node[vertex] (v6) at (3,0){};
	\node[vertex] (u1) at (1,1){};
	\node[vertex] (u2) at (2,1){};
	\draw[edge](v1)--(v2);
	\draw[edge](v2)--(v3);
	\draw[edge](v4)--(v5);
	\draw[edge](v5)--(v6);
	\draw[edge](v2)--(u1);
	\draw[edge](u2)--(v5);
	\draw[edge](u1)--(1.3,1);
	\node[color = gray](dots) at (1.65,1){$\dots$};	
	\draw[br](2.925,0.875)--(0.075,0.875);
	\scriptsize
	\node[](i) at (1.5, 0.6){$i$ edges};
\end{scope}
\end{tikzpicture}
}
\caption{Left: a graph with a matching cut (i.e., a $1$-cut). Middle: a graph with a $3$-cut but no $d$-cut for $d\leq 2$. Right: the graph $H_i^*$.}\label{f-examplesdcut}
\vspace*{-0.2cm}
\end{figure}

Graphs with matching cuts were introduced in the context of number theory~\cite{Gr70} 
and have various other applications~\cite{ACGH12,FP82,GPS12,PP01}.
 The {\sc Matching Cut} problem is to decide if a connected graph has a matching cut. This problem was shown to be \NP-complete by Chv\'atal~\cite{Ch84}.
Several variants and generalizations of matching cuts are known. In particular, a {\it perfect matching cut} is a matching cut that is a perfect matching, whereas a {\it disconnected perfect matching} is a perfect matching containing a matching cut. The corresponding decision problems  {\sc Perfect Matching Cut}~\cite{HT98} and {\sc Disconnected Perfect Matching}~\cite{BP}
are also \NP-complete; see~\cite{BCD24,FLPR25,LL23,LT22,LPR23a} for more complexity results for these two problems. The optimization versions {\sc Maximum Matching Cut} and {\sc Minimum Matching Cut} are to find a matching cut of maximum and minimum size in a connected graph, respectively; see~\cite{LLPR24,LPR24} for more details.

\medskip
\noindent
{\bf Our Focus.}
Matching cuts have also been generalized as follows.
For an integer~$d\geq 1$ and a connected graph $G=(V,E)$, a set $M\subseteq E$ is a {\it $d$-cut} of $G$ if it is possible to partition $V$ into two non-empty sets $B$ and $R$, such that: (i) the set~$M$ is the set of all edges with one end-vertex in $B$ and the other one in $R$; and (ii) every vertex in $B$ has at most $d$ neighbours in $R$, and vice versa  (see also Figure~\ref{f-examplesdcut}). Note that a $1$-cut is a matching cut. 
We consider the {\sc $d$-Cut} problem: does a connected graph have a $d$-cut? Here, $d\geq 1$ is a fixed integer, so not part of the input.  Note that $1$-{\sc Cut} is {\sc Matching Cut}. The {\sc $d$-Cut} problem was introduced by Gomes and Sau~\cite{GS21} who proved its \NP-completeness for all~$d\geq 1$.

\medskip
\noindent
{\bf Our Goal.}
To get a better understanding of the hardness of an \NP-complete graph problem, it is natural to restrict the input to belong to some special graph classes. We will first give a brief survey of the known complexity results for {\sc Matching Cut} and {\sc $d$-Cut} for $d\geq 2$ under input restrictions. As we will see, for {\sc Matching Cut} many more results are known than for {\sc $d$-Cut} with $d\geq 2$.
Our goal is to obtain the same level of understanding of the {\sc $d$-Cut} problem for $d\geq 2$. This requires a currently lacking systematic study into the complexity of this problem. We therefore consider the following research question:

\medskip
\noindent
{\it For which graph classes~${\cal G}$ does the complexity of {\sc $d$-Cut}, restricted to graphs from ${\cal G}$, change if $d\geq 2$ instead of $d=1$?}

\medskip
\noindent
As {\it testbeds} we take classes of graphs of bounded diameter, graphs of bounded radius, and $H$-free graphs. 
The {\it distance} between two vertices $u$ and $v$ in a connected graph~$G$ is the {\it length} (number of edges) of a shortest path between $u$ and $v$ in~$G$. 
The {\it eccentricity} of a vertex $u$ is the maximum distance between $u$ and any other vertex of $G$. The {\it diameter} of $G$ is the maximum eccentricity over all vertices of~$G$, whereas the {\it radius} of $G$ is the minimum eccentricity over all vertices of $G$. A graph~$G$ is {\it $H$-free} if $G$ does not contain a graph $H$ as an {\it induced subgraph}, that is, $G$ cannot be modified into $H$ by vertex deletions. 

\medskip
\noindent
{\bf Existing Results.}
We focus on classical complexity results; see~\cite{AS21,GS21} for exact and parameterized complexity results for \dcut.
Let $K_{1,r}$ denote the $(r{+}1)$-vertex star, which has vertex set $\{u,v_1,\ldots,v_r\}$ and edges $uv_i$ for $i\in \{1,\ldots,r\}$.
 Chv\'atal~\cite{Ch84} showed that {\sc Matching Cut} is \NP-complete even for $K_{1,4}$-free graphs of maximum degree~$4$, but 
 polynomial-time solvable for graphs of maximum degree at most~$3$. Gomes and Sau~\cite{GS21} extended these results by proving that for every $d\geq 2$,  \dcut{} is \NP-complete for
 $(2d+2)$-regular graphs, 
 but 
 polynomial-time solvable for graphs of maximum degree at most $d+2$. 
Feghali et al.~\cite{FLPR25} 
proved
that for every $d\geq 1$ and every $g \geq 3$, there is a function~$f(d)$,
 such that {\sc $d$-Cut} is \NP-complete for bipartite graphs of girth at least~$g$ and maximum degree at most $f(d)$.
 The {\it girth} of a graph~$G$ that is not a forest is the length of a shortest induced cycle in $G$. 
It is also known that {\sc Matching Cut} is polynomial-time solvable for graphs of diameter at most~$2$~\cite{BJ08,LL19}, and even radius at most~$2$~\cite{LPR22}, while being \NP-complete for graphs of diameter~$3$~\cite{LL19}, and thus radius at most~$3$ 
(for a survey of other results, see~\cite{CHLLP21}). 
 Hence, we obtain:

\begin{theorem}[\cite{LL19,LPR22}]\label{t-diameter}
For $r\geq 1$, {\sc Matching Cut} is polynomial-time solvable for graphs of diameter~$r$ and graphs of radius~$r$ if $r\leq 2$ and \NP-complete if $r\geq 3$.
\end{theorem}

\noindent
To study a problem in a systematic way on graph classes that can be characterized by forbidden induced subgraphs, an often used approach is to first focus on the classes of $H$-free graphs.
As {\sc Matching Cut} is \NP-complete for graphs of girth~$g$ for every $g\geq 3$~\cite{FLPR25} and for $K_{1,4}$-free graphs~\cite{Ch84}, {\sc Matching Cut} is \NP-complete for $H$-free graphs whenever $H$ has a cycle or is a forest with a vertex of degree at least~$4$.
What about when $H$ is a forest of maximum degree~$3$?

We let $P_t$ be the path on $t$ vertices. We denote the disjoint union of two vertex-disjoint graphs $G_1+G_2$ by $G_1+G_2=(V(G_1)\cup V(G_2),E(G_1)\cup E(G_2))$. We let $sG$ be the disjoint union of $s$ copies of $G$.
Feghali~\cite{Fe23} proved the existence of an integer~$t$ such that {\sc Matching Cut} is \NP-complete for $P_t$-free graphs, which was narrowed down to  $(3P_5,P_{15})$-free graphs in~\cite{LPR23a} and to  $(3P_6,2P_7,P_{14})$-free graphs in~\cite{LL23}. Let $H^*_1$ be the ``H''-graph, which has vertices $u,v,w_1,w_2,x_1,x_2$ and edges $uv, uw_1,uw_2,vx_1,vx_2$.
For $i\geq 2$, let $H_i^*$ be the graph obtained from $H_1^*$ by subdividing $uv$ exactly $i-1$ times; see  Figure~\ref{f-examplesdcut}. It is known that {\sc Matching Cut} is \NP-complete for $(H_1^*,H_3^*,H_5^*,\ldots)$-free bipartite graphs~\cite{Mo89} and for $(H_1^*,\ldots,H_i^*)$-free  graphs for every $i\geq 1$~\cite{FLPR25}.

On the positive side, {\sc Matching Cut} is polynomial-time solvable  for {\it claw-free} graphs ($K_{1,3}$-free graphs) and 
for $P_6$-free graphs~\cite{LPR23a}.
Moreover, if {\sc Matching Cut} is polynomial-time solvable for $H$-free graphs for some graph~$H$, then it is so for $(H+P_3)$-free graphs~\cite{LPR23a}. 

For two graphs $H$ and $H'$, we write $H\ssi H'$ if $H$ is an induced subgraph of~$H'$. Combining the above yields a partial classification (see also~\cite{FLPR25,LPR24}):

\begin{theorem}[\cite{Bo09,Ch84,FLPR25,LL23,LPR22,LPR23a,Mo89}]\label{t-s0}
For a graph~$H$, {\sc Matching Cut} on $H$-free graphs is 
\begin{itemize}
\item polynomial-time solvable if $H\ssi sP_3+K_{1,3}$ or $sP_3+P_6$ for some $s\geq 0$;
\item \NP-complete if $H\si K_{1,4}$, $P_{14}$, $2P_7$, $3P_5$, $C_r$ for some $r\geq 3$, or $H_i^*$ for some $i\geq 1$.
\end{itemize}
\end{theorem}

\noindent
{\bf Our Results.}
We first note that
also for every $d\geq 2$,
 {\sc $d$-Cut} is straightforward to solve for graphs of radius~$1$ (i.e., graphs with a dominating vertex). 
We will prove that for every $d\geq 2$, {\sc $d$-Cut} is also polynomial-time solvable for diameter~$2$ but \NP-complete for graphs of diameter~$3$ and radius~$2$. This leads to the following extensions of Theorem~\ref{t-diameter}:

\begin{theorem}\label{t-ddiameter}
Let $d\geq 2$. For $r\geq 1$, {\sc $d$-Cut} is polynomial-time solvable for graphs of diameter~$r$ if $r\leq 2$ and \NP-complete if $r\geq 3$.
\end{theorem}

\begin{theorem}\label{t-dradius}
Let $d\geq 2$. For $r\geq 1$, {\sc $d$-Cut} is polynomial-time solvable for graphs of radius~$r$ if $r\leq 1$ and \NP-complete if $r\geq 2$.
\end{theorem}

\noindent
Comparing Theorem~\ref{t-diameter} with Theorems~\ref{t-ddiameter} and~\ref{t-dradius} shows no difference in complexity for diameter but a complexity jump from $d=1$ to $d=2$ for radius.

For $d\geq 2$, we also give polynomial-time algorithms for {\sc $d$-Cut} for $(P_3+P_4)$-free graphs and $P_5$-free graphs. 
Our proof techniques use novel arguments, as we can no longer rely on a polynomial-time algorithm for radius~$2$ or a reduction to~$2$-SAT as for $d=1$~\cite{LL19,LPR22}.
Moreover, we show that for $d\geq 2$, \dcut{} is polynomial-time solvable for $(H+P_1)$-free graphs whenever \dcut{} is so for $H$-free graphs, thus the cases $\{H+sP_1\; |\; s\geq 0\}$ are all {\it (polynomially) equivalent}. All these results extend the known results for $d=1$, as can be seen from Theorem~\ref{t-s0}.

As negative results, we prove that
for all $d\geq 2$,  {\sc $d$-Cut} is \NP-complete for $3P_2$-free graphs, and
 that 
for every $d\geq 3$,
\dcut{} is \NP-complete 
for line graphs, and thus for $K_{1,3}$-free graphs.
Recently, Ahn et al.~\cite{AELPS25}  proved that $2$-{\sc Cut} is \NP-complete for $K_{1,3}$-free graphs. 
The \NP-completeness for graphs of large girth from~\cite{FLPR25} implies that for $d\geq 2$, {\sc $d$-Cut} is \NP-complete for $H$-free graphs if $H$ has a cycle.
Hence, by combining the above results, we obtain the following 
partial complexity classification
for $d\geq 2$:

\begin{theorem}\label{t-s2}
Let $d\geq 2$. For a graph~$H$, {\sc $d$-Cut} on $H$-free graphs is 
\begin{itemize}
\item polynomial-time solvable if $H\ssi  sP_1+P_3+P_4$ or $sP_1+P_5$ for some $s\geq 0$;
\item \NP-complete if $H\si K_{1,3}$, $3P_2$, or $C_r$ for some $r\geq 3$.
\end{itemize}
\end{theorem}

\noindent 
Theorem~\ref{t-s2} leaves
only three non-equivalent open cases for every $d\geq 2$, namely when $H=2P_4$, $H= P_6$ and $H=P_7$.
From Theorems~\ref{t-s0} and~\ref{t-s2} we observe that there are complexity jumps from $d=1$ to $d=2$ for $sP_2$-free graphs when $s=3$ 
and for $K_{1,3}$-free graphs.

We prove our polynomial-time results in Section~\ref{s-poly} and our \NP-completeness results in Section~\ref{s-np}. We finish our paper with 
 some
open problems in Section~\ref{s-con}. We start with providing some basic results in Section~\ref{s-prelim}.

\section{Preliminaries}\label{s-prelim}

Throughout the paper, we only consider finite, undirected graphs without multiple edges and self-loops. We first define some general graph terminology.

Let $G=(V,E)$ be a graph. 
The \emph{line graph} $L(G)$ of $G$ has the edges of $G$ as its vertices, with an edge between two vertices in $L(G)$ if and only if the corresponding edges in $G$ share an end-vertex. Let $u\in V$. The set $N(u)=\{v \in V\; |\; uv\in E\}$ is the {\it neighbourhood} of $u$, and $|N(u)|$ is the {\it degree} of $u$.  Let $S\subseteq V$. 
The {\it (open) neighbourhood} of $S$ is the set $N(S)=\bigcup_{u\in S}N(u)\setminus S$, and the {\em closed neighbourhood} $N[S] = N(S) \cup S$.
We let $G[S]$ denote the subgraph of $G$ {\it induced} by $S$, which is obtained from $G$ by deleting the vertices not in $S$. We let $G-S=G[V\setminus S]$.
If no two vertices in $S$ are adjacent, then $S$ is an {\it independent set} of $G$. If every two vertices in $S$ are adjacent, then $S$ is a {\it clique} of $G$. 
If every vertex of $V\setminus S$ has a neighbour in~$S$, then $S$ is a {\it dominating} set of~$G$. We also say that $G[S]$ {\it dominates} $G$. The {\it domination number} of $G$ is the size of a smallest dominating set of $G$. 
Let $T\subseteq V\setminus S$. The sets $S$ and $T$ are \emph{complete} to each other if every vertex of $S$ is adjacent to every vertex of $T$.

Let $u$ be a vertex in a connected graph $G$. We denote the distance of $u$ to some other vertex $v$ in $G$ by $\dist_G(u,v)$.
Recall that the eccentricity of $u$ is the maximum distance between $u$ and any other vertex of $G$. Recall also that the diameter $\diam(G)$ of $G$ is the maximum eccentricity over all vertices of~$G$ and that
the radius $\radius(G)$ of $G$ is the minimum eccentricity over all vertices of~$G$. Note that $\radius(G)\leq \diam(G)\leq 2\cdot\radius(G)$.

We now generalize some colouring terminology that was used in the context of matching cuts (see, e.g.,~\cite{Fe23,LPR22}). 
A {\it red-blue colouring} of a graph $G$ assigns every vertex of $G$ either the colour red or blue.
For $d\geq 1$, a red-blue colouring is a {\it red-blue $d$-colouring} if every blue vertex has at most $d$ red neighbours; every red vertex has at most $d$ blue neighbours; and both colours red and blue are used at least once. See Figure~\ref{f-examplesdcut} for examples of a red-blue $1$-colouring and a red-blue $3$-colouring.

We make the following observation.

\begin{observation}\label{o-cut-colouring}
For every $d\geq 1$, a connected graph $G$ has a $d$-cut if and only if it has a red-blue $d$-colouring.
\end{observation}

If every vertex of a set $S\subseteq V$ has the same colour (either red or blue) in a red-blue colouring, then $S$, and also $G[S]$, are  {\it monochromatic}. An edge with a blue and a red end-vertex is \emph{bichromatic}.
Note that for every $d\geq1$,  the graph $K_{2d+1}$ is monochromatic in every red-blue $d$-colouring, and that every connected graph with a red-blue $d$-colouring contains a bichromatic edge.

We now generalize a known lemma for {\sc Matching Cut} (see, e.g., \cite{LPR22}).

\begin{lemma}\label{l-smalldom}
For every $d,g\geq 1$, it is possible to find in $O(2^gn^{dg+2})$-time a red-blue $d$-colouring (if it exists) of a graph $G$ with $n$ vertices and domination number~$g$.
\end{lemma}

\begin{proof}
Let $d,g\geq 1$ and $G$ be a graph on $n$ vertices with domination number~$g$.
Let $D$ be a dominating set~$D$ of $G$ that has size at most $g$. 

We consider all $2^{|D|}\leq 2^g$ options of giving the uncoloured vertices of $D$ either colour red or blue. For each red-blue colouring of $D$ we do as follows.
For every red vertex of $D$, we consider all $O(n^d)$ options of colouring at most~$d$ of its uncoloured neighbours blue, and we colour all of its other uncoloured neighbours red.
Similarly, for every blue vertex of $D$, we consider all $O(n^d)$ options of colouring at most $d$ of its uncoloured neighbours red, and we colour all of its other uncoloured neighbours blue.
As $D$ dominates $G$, we obtained a red-blue colouring $c$ of the whole graph $G$. We discard the option if $c$ is not a red-blue 
$d$-colouring
of~$G$.

We note that any red-blue $d$-colouring
of $G$, if it exists, will be found by the above algorithm. As the total number of options is $O(2^gn^{dg})$ and checking if a red-blue colouring is a red-blue $d$-colouring
takes $O(n^2)$ time, our algorithm has total running time $O(2^gn^{dg+2})$.
\qed
\end{proof}

Let $d\geq 1$. Let $G=(V,E)$ be a connected graph and $S, T \subseteq V$ be two disjoint sets. A \emph{red-blue $(S, T)$-$d$-colouring} of $G$ is a red-blue $d$-colouring of $G$ that colours all the vertices of $S$ red and all the vertices of $T$ blue. We call $(S,T)$ a {\it precoloured pair} of $(G,d)$ which is {\it colour-processed} if every vertex of $V\setminus (S\cup T)$ is adjacent to at most $d$ vertices of $S$ and to at most $d$ vertices of~$T$. 

By the next lemma, we may assume without loss of generality that a precoloured pair $(S,T)$ is always colour-processed.

\begin{lemma}\label{l-process}
Let $G$ be a connected graph with a precoloured pair $(S,T)$. It is possible, in polynomial time, to either colour-process $(S,T)$ 
or to find that $G$ has no red-blue $(S,T)$-$d$-colouring.
\end{lemma}

\begin{proof}
We apply the following rules on $G$. Let $Z=V\setminus (S\cup T)$. If $v$ is adjacent to $d+1$ vertices in $S$, then move $v$ from $Z$ to $S$. If $v$ is adjacent to $d+1$ vertices in $T$, then move $v$ from $Z$ to $T$. Return {\tt no} if a vertex $v \in Z$ at some point becomes adjacent to $d+1$ vertices in $S$ as well as to $d+1$ vertices in $T$. We apply these three rules exhaustively.  It is readily seen that each of these rules is safe to use, and moreover, can be verified and applied in polynomial time. After each application of a rule, we either stop or have decreased the size of $Z$ by at least one vertex. Hence, the procedure is correct and takes polynomial time. \qed
\end{proof}

\section{Polynomial-Time Results}\label{s-poly}

We now show our polynomial-time results for \dcut{} for $d\geq 2$. These results complement the corresponding known polynomial-time results for $d=1$ that we mentioned in Section~\ref{s-intro}
(and the proofs of the new results yield alternative proofs for the known results if we set $d=1$).

We first consider the class of graphs of diameter at most~$2$.

\begin{theorem}\label{t-diam}
For every $d\geq 2$, the \dcut{} problem is polynomial-time solvable for graphs of diameter at most~$2$.
\end{theorem}

\begin{proof}
Let $G=(V,E)$ be a graph of diameter at most~$2$. Note that this implies that $G$ is connected. As \dcut{} is trivial for graphs of diameter~$1$, we assume that $G$ has diameter~$2$.
By Observation~\ref{o-cut-colouring}, we have proven the theorem if we can decide in polynomial time if $G$ has a red-blue $d$-colouring.

Let $v\in V$. Without loss of generality we may colour $v$ red.
Since $v$ has at most $d$ blue neighbours in any red-blue $d$-colouring, we can branch over all $O(n^{d})$ options to colour the neighbourhood of $v$.
In the case where all neighbours of $v$ are red, we branch over all $O(n)$ options to colour some vertex~$x$ of $G$ blue.
We consider each option separately.

Let $R$ and $B$ be the set of red and blue vertices, respectively.
 Let $S' = \{v\} \cup \{u\; |\; u \in N(v), u \in R\}$ and let $T' =  \{u\; |\; u \in N(v), u \in B\}$. We colour-process $(S',T')$, resulting in a pair $(S,T)$.
Let $Z = V\setminus (S \cup T)$ be the set of uncoloured vertices.
As $(S,T)$ is colour-processed, every vertex in $Z$ has at most $d$ red neighbours and at most $d$ blue neighbours and thus in total at most $2d$ coloured neighbours. 
We distinguish between the following two cases:

\medskip
\noindent
{\bf Case 1.} $G[Z]$ consists of at least two connected components.\\
Let $Z_1, \dots, Z_r$ be the connected components of $G[Z]$. By assumption, we have that $r\geq 2$.
Let $v_1 \in V(Z_1)$ and $v_2 \in V(Z_2)$.
Since $G$ has diameter~$2$, we know that $v_1$ has a common neighbour with every vertex in $Z_2, \dots, Z_r$. 
Let $N_1$ be the set of all vertices that are adjacent to $v_1$ as well as to a vertex in $V(Z_2)\cup \cdots \cup V(Z_r)$.
Note that $N_1 \subseteq S \cup T$ and that every vertex of $V(Z_2)\cup \cdots \cup V(Z_r)$ has at least one neighbour in $N_1$.
Moreover, as $v_1$ has at most $2d$ coloured neighbours and $N_1 \subseteq S \cup T$, we find that $|N_1| \leq 2d$.
By the same arguments, we find a set $N_2\subseteq S\cup T$, with $|N_2| \leq 2d$ consisting of the common neighbours of $v_2$ and $Z_1$, such that every vertex of $Z_1$ has at least one neighbour in $N_2$. See also Figure~\ref{f-diam2}.

In a red-blue $d$-colouring, every vertex in $N_1\cup N_2$ has at most $d$ neighbours of the other colour. 
Thus, we can try all $O(n^{4d^2})$ options to colour the uncoloured neighbours of $N_1\cup N_2$.
If in the resulting colouring any vertex has more than $d$ neighbours of the other colour we discard it. Suppose not. Then, as every vertex of $V(Z_1)\cup \cdots \cup V(Z_r)$ is a neighbour of at least one vertex in $N_1\cup N_2$, we found a red-blue $d$-colouring of $G$.

\begin{figure}[b]
\centering
\begin{tikzpicture}

\tikzstyle{comp}=[,draw, ,thick, minimum width = 35pt, minimum height = 30pt]
\tikzstyle{bigcomp}=[draw, thick, minimum width = 80pt, minimum height = 40pt]
\tikzstyle{rbset}=[draw, ,thick, minimum width = 55pt, minimum height = 40pt]
\tikzstyle{n1}=[draw, gray, thick, , minimum width = 40pt, minimum height = 16pt]
\tikzstyle{U}=[draw, gray, thick, , minimum width = 28pt, minimum height = 38pt]

\normalsize
\begin{scope}[]
\node[rbset, lightblue, label={:$T$}](b) at (1,2){};
\node[rbset, nicered, label={:$S$}](r) at (3,2){};
\node[rvertex, label=right:$v$](v) at (2.3, 2.3){};
\node[n1, label=left:$N_1$](n) at (2,1.6){};
\node[bvertex](u1) at (1.5, 1.6){};
\node[bvertex](u2) at (1.8, 1.6){};
\node[rvertex](u3) at (2.4, 1.6){};

\node[comp, label=below:$Z_1$](z1) at (0,0){};
\node[comp, label=below:$Z_2$](z2) at (2,0){};
\node[comp, label=below:$Z_3$](z3) at (4,0){};

\node[vertex, label=left:$v_1$](v1) at (0,0){};
\node[vertex](v2) at (2,0.3){};
\node[vertex](v3) at (4,-0.2){};

\draw[edge](v1) -- (u1);
\draw[edge](v1) -- (u2);
\draw[edge](v1) -- (u3);
\draw[edge](v2) -- (u1);
\draw[edge](v2) -- (u3);
\draw[edge](v3) -- (u2);

\end{scope}

\begin{scope}[shift = {(6,0)}]
\node[rbset, lightblue, label={:$T$}](b) at (1,2){};
\node[rbset, nicered, label={:$S$}](r) at (3,2){};
\node[rvertex, label=right:$v$](v) at (2.3, 2.3){};

\node[bigcomp, label=below:$Z$](Z) at (2,0){};
\node[vertex, label=below:$z$](z) at (1,0){};

\node[bvertex](u1) at (1.5, 1.6){};
\node[bvertex](u2) at (1.8, 1.6){};
\node[rvertex](u3) at (2.4, 1.6){};

\node[vertex](z1) at (1.5,0.3){};
\node[vertex](z2) at (1.5,-0.3){};

\node[vertex](z3) at (2,0.5){};
\node[vertex,](z4) at (2,0){};
\node[vertex](z5) at (2,-0.5){};

\node[U] (U) at (2.87,0){};
\node[vertex](v1) at (2.6,0.4){};
\node[vertex](v2) at (2.6,-0.2){};
\node[vertex](v3) at (3.2,0.2){};
\node[](Utext) at (3.1,-0.45){$U$};

\draw[edge](z) -- (u1);
\draw[edge](z) -- (u2);
\draw[edge](z) -- (u3);
\draw[edge](v2) -- (u1);
\draw[edge](v1) -- (u3);
\draw[edge](v3) -- (u2);

\draw[edge](z) -- (z1);
\draw[edge](z) -- (z2);
\draw[edge](z1) -- (z3);
\draw[edge](z1) -- (z4);
\draw[edge](z2) -- (z4);
\draw[edge](z2) -- (z5);
\draw[edge](z3) -- (z4);

\draw[edge](z3) -- (v1);
\draw[edge](z4) -- (v2);
\draw[edge](v2) -- (v3);

\end{scope}

\end{tikzpicture}
\caption{Cases 1 (left) and 2 (right) in the proof of Theorem~\ref{t-diam}.}\label{f-diam2}
\end{figure}

\medskip
\noindent
{\bf Case 2.} $G[Z]$ is connected.\\
We first assume that $G[Z]$ has radius at least~$3$.
We let $z \in Z$, and we let $$U = \{u \in Z : \dist_{G[Z]}(z,u) \geq 3\}.$$ Observe that $U \neq \emptyset$, as $G[Z]$ has radius at least $3$. For an illustration of $U$, see Figure~\ref{f-diam2}.
As $G$ has diameter~$2$, we find that $z$ has a common neighbour with every vertex of $U$. As the common neighbours of $z$ and the vertices in $U$ are not in $Z$, they belong to $S \cup T$, meaning they are coloured red or blue. As $z$ was not coloured after the colour-processing, we find that $z$ has at most $2d$ coloured neighbours. Each of these coloured neighbours of $z$ can have at most $d$ neighbours of the other colour. So we can branch over all $O(n^{2d^2})$ options to colour $U$. 

For each of the $O(n^{2d^2})$ branches, we do as follows.
If some vertex of $U$, which is now coloured, has more than $d$ neighbours of the other colour, then we discard the branch.
Else, we consider the graph $G[Z \setminus U]$, which consists of all uncoloured vertices and which has radius at most~$2$. Let $Z := Z \setminus U$.

Let $z \in Z$ such that $\dist_{G[Z]}(z,u) \leq 2$ for all $u \in Z$.
We branch over all $O(n^d)$ colourings of $\{z\} \cup N_{G[Z]}(z)$.
As $\dist_{G[Z]}(z,u) \leq 2$ for all $u \in Z$, it holds that  $\{z\} \cup N_{G[Z]}(z)$ is a dominating set of $G[Z]$.
Hence, every uncoloured vertex has now at least one newly coloured neighbour. 

We now colour-process the pair of red and blue sets. If there is a vertex with $d+1$ neighbours of the other colour or an uncoloured vertex with more than $d$ neighbours of each colour, then we discard the branch. Otherwise, we redefine $Z$ to be the new set of uncoloured vertices.

\medskip
\noindent
We now proceed by checking if Case~1 applies. If so, we apply the algorithm under Case~1, and else we proceed according to Case~2 again.
Note that if at some point we apply Case~1, then we are done: we either found a red-blue $d$-colouring of $G$, or we have discarded the branch.
We also recall that every time we apply Case~2, all vertices that remain uncoloured will have at least one newly coloured neighbour. This observation is crucial, as it means that
we only need to apply Case 2 at most $2d$ times (should we apply Case~2 at some point $2d$ times,  every uncoloured vertex will have at least $2d$ coloured neighbours and thus, they are coloured after we colour-processed the sets of red and blue vertices).

The correctness of our algorithm follows from its description. We now discuss its running time.
We always have $O(n^d)$ branches at the start of our algorithm when we colour a specific vertex $v$ and its neighbourhood (and possibly one more vertex~$x$). In each of these branches, we may colour-process once, which takes polynomial time by Lemma~\ref{l-process}.
One application of Case 1 gives $O(n^{4d^2})$ branches and one application of Case 2 gives $O(n^{2d^2 + d})$ branches. 
Recall that the algorithm terminates after one application of Case 1 and that we apply Case~2 at most $2d$ times. This means that the total number of branches is  
$$O\left(n^d \cdot \left(n^{2d^2+d}\right)^{^{2d}} \cdot n^{^2}\right),$$ which is a polynomial number and which also bounds the number of times we colour-process in a branch. 
As the latter takes polynomial time by Lemma~\ref{l-process}, we conclude that our algorithm runs in polynomial time.
\qed
\end{proof}

We now consider two classes of $H$-free graphs, starting with the case $H=P_5$.

\begin{theorem}\label{t-p5}
For every $d\geq 2$, the \dcut{} problem is polynomial-time solvable for $P_5$-free graphs.
\end{theorem}

\begin{proof}
Let $d\geq 2$. Let $G=(V,E)$ be a connected $P_5$-free graph on $n$ vertices. As $G$ is $P_5$-free and connected, $G$ has a dominating set $D$, such that either $D$ induces a cycle on five vertices or $D$ is a clique~\cite{LZ94}. Moreover, we find such a dominating set $D$ in $O(n^3)$  time~\cite{HP10}. If $|D|\leq 3d$, then we apply Lemma~\ref{l-smalldom}. Now assume that $|D|\geq 3d+1\geq 7$, and thus $D$ is a clique.  
If $D=V$, then $G$ has no $d$-cut, as $|D|\geq 3d+1$ and $D$ is a clique. Assume that $D\subsetneq V$, so $G-D$ has at least one connected component. 
Below we explain how to find in polynomial time a $d$-cut of $G$, or to conclude that $G$ does not have a $d$-cut. By Observation~\ref{o-cut-colouring}, we need to decide in polynomial time if $G$ has a red-blue $d$-colouring.

We first enter the {\it blue phase}  of our algorithm. As $|D|\geq 3d+1$ and $D$ is a clique, $D$ is monochromatic in any red-blue $d$-colouring of $G$. Hence, we may colour, without loss of generality, all the vertices of $D$ blue, and we may colour all vertices of every connected component of $G-D$ except one connected component blue as well. We branch over all $O(n)$ options of choosing the connected component $L_1$ of $G-D$ that will contain a red 
vertex.\footnote{For $d=1$, up to now, the same approach is used for $P_5$-free graphs~\cite{Fe23}. But the difference is that for $d=1$, the algorithm and analysis is much shorter: one only has to check, in this stage, if there is a component of $G-D$, whose vertices can all be safely coloured red. Then one either finds a matching cut, or $G$ has no matching cut.} 
For each option, we are going to repeat the process of colouring vertices blue until we colour at least one vertex red. We first explain how this process works if we do not do this last step.

As $L_1$ is connected and $P_5$-free, we find in $O(n^3)$ time (using the algorithm of~\cite{HP10}) a dominating set $D_1$ of $L_1$ that is either a cycle on five vertices or a clique. We colour the vertices of $D_1$ blue. As we have not used the colour red yet, we colour all vertices of every connected component of $L_1-D_1$ except one connected component blue. So, we branch over all $O(n)$ options of choosing the connected component $L_2$ of $L_1-D_1$ that will contain a red vertex. Note that every uncoloured vertex of $G$, which belongs to $L_2$, has both a neighbour in $D$ (as $D$ dominates $G$) and a neighbour in $D_1$ (as $D_1$ dominates $L_1$). We now find a dominating set $D_2$ of $L_2$ and a connected component $L_3$ of $G-D_2$ with a dominating set $D_3$, and so on. See also Figure~\ref{f-p5}.

\begin{figure}[t]
\centering
\scalebox{0.9}{
\begin{tikzpicture}
\def\e{0.05}
\large
\begin{scope}[]

\draw[lightblue,   thick](0,4) rectangle (6 ,5);
\node[](d) at (5.5, 4.5){$D$};
\draw[thick](-0.1 ,0 ) rectangle (3.1, 3.6);
\node[](d) at (2.7,  0.3){$L_1$};

\draw[lightblue,  thick] (3.3, 2.5 ) rectangle (4.3, 3.5 );
\node[] (d) at (4.65, 3){$\dots$};
\draw[lightblue,  thick] (5 , 2.5) rectangle (6, 3.5 );

\draw[lightblue,  thick] (0 , 2.5) rectangle (3, 3.5 );
\node[](d) at (2.5, 3){$D_1$};

\draw[thick](0.1,0.1) rectangle (1.9, 2.1);
\node[](d) at (1.6,  0.4){$L_2$};

\draw[lightblue,  thick] (2.1, 1.5 ) rectangle (2.4, 2);
\node[] (d) at (2.55, 1.75){{\tiny ...}};
\draw[lightblue,  thick] (2.7,  1.5 ) rectangle (3, 2);

\draw[lightblue,  thick] (0.2,  1.4 ) rectangle (1.8, 2);
\node[](d) at (1,  1.7){$D_2$};

\draw[edge](1.5, 4) -- (1.5, 3.6);
\draw[edge](3.8, 4) -- (3.8, 3.5);
\draw[edge](5.5, 4) -- (5.5, 3.5);

\draw[edge](1, 2.5) -- (1, 2.1);
\draw[edge](2.25, 2.5) -- (2.25, 2);
\draw[edge](2.85, 2.5) -- (2.85, 2);

\end{scope}

\begin{scope}[shift = {(7,0)}]
\draw[lightblue,  thick](0,4) rectangle (4 ,5);
\node[] (D) at (0.5, 4.5){$D$};

\draw[thick](1 ,0.7) rectangle (5, 3.6);
\draw[nicered, thick] (1.8,2.4) rectangle (4.9, 3.4);
\node[](d) at (4.6, 2.9){$D_i$};

\node[](d) at (1.3, 1){$L_i$};

\node[bvertex, label= left:$z$](z) at (2,4.5){};
\node[bvertex, label = right: $y_1$](y1) at (3,4.5){};
\node[rvertex, label = left: $x_j$](xj) at (2.5, 2.9){};
\node[rvertex, label = right: $x_1$](x1) at (3.5, 2.9){};
\node[vertex, label = left: $w_j$](wj) at (2.5, 1.5){};
\node[vertex, label = right: $w_1$](w1) at (4, 1.5){};

\draw[tedge](z) -- (y1);
\draw[tedge](y1) -- (x1);
\draw[tedge](xj) -- (x1);
\draw[tedge](xj) -- (wj);
\draw[edge](x1) -- (w1);

\end{scope}
\end{tikzpicture}}
\caption{The graph $G$ in the blue phase (left) and in the red phase: Case 2 (right).}\label{f-p5}
\end{figure}

If we repeat the above process more than $d$ times, we have either coloured every vertex of $G$ blue, or we found $d+1$ pairwise disjoint, blue sets $D,D_1,\ldots,D_d$, such that every uncoloured vertex~$u$ has a neighbour in each of them. The latter implies that $u$ has $d+1$ blue neighbours, so $u$ must be coloured blue as well. Hence, in each branch, we would eventually end up with the situation where all vertices of $G$ will be coloured blue. To prevent this from happening, we must colour, in each branch, at least one vertex of $D_i$ red, for some $1\leq i\leq d-1$. As soon as we do this, we end the blue phase for the branch under consideration, and our algorithm enters the {\it red phase}. 

Each time we have $O(n)$ options to select a connected component $L_i$ and, as argued above, we do this at most $d$ times. Hence, we enter the red phase for $O(n^d)$ branches in total. From now on, we call these branches the {\it main branches} of our algorithm. For a main branch, we say that we {\it quit the blue phase at level~$i$} if we colour at least one vertex of $D_i$ red. If we quit the blue phase for a main branch at level~$i$, for some $1\leq i\leq d-1$, then we have constructed, in polynomial time, pairwise disjoint sets $D$, $D_1,\ldots, D_i$ and graphs $L_1,\ldots,L_i$, such that: 

\begin{itemize}
\item for every $h\in \{1,\ldots,i-1\}$, $L_{h+1}$ is a connected component of $L_h-D_h$;
\item every vertex of $G$ that does not belong to $L_i$ has been coloured blue;
\item for every $h\in \{1,\ldots,i\}$, $D_h$ induces a cycle on five vertices or is a clique;
\item $D$ dominates $G$, so, in particular, $D$ dominates $L_i$; and
\item for every $h\in \{1,\ldots,i\}$, $D_h$ dominates $L_h$.
\end{itemize}

\noindent
 We now prove the following claim, which shows that we branched correctly.

\begin{nclaim}\label{c-cor}
The graph $G$ has a red-blue $d$-colouring that colours every vertex of $D$ blue if and only if we have quit a main branch at level~$i$ for some $1\leq i\leq d-1$, such that $G$ has a red-blue $d$-colouring that colours at least one vertex of $D_i$ red and all vertices not in $L_i$ blue.
\end{nclaim}

\begin{claimproof}
Suppose $G$ has a red-blue $d$-colouring $c$ that colours every vertex of $D$ blue (the reverse implication is immediate). By definition, $c$ has coloured at least one vertex~$u$ in $G-D$ red. As we branched in every possible way, there is a main branch that quits the blue phase at level~$i$, such that $u$ belongs to $L_i$ for some $1\leq i\leq d-1$. We pick a main branch with largest possible~$i$, so $u$ is not in $L_{i+1}$. Hence, $u\in D_i$. We also assume that no vertex~$u'$ that belongs to some $D_h$ with $h<i$ is coloured red by~$c$, as else we could take $u'$ instead of $u$. Hence, $c$ has coloured every vertex in $D_1\cup \ldots \cup D_{i-1}$ (if $i\geq 2$) blue. Therefore, we may assume that every vertex~$v\notin D\cup D_1 \cup \ldots \cup D_{i-1}\cup V(L_i)$ has  been coloured blue by~$c$. If not, may just recolour all such vertices~$v$ blue for the following reasons: $u$ is still red and no neighbour of $v$ is red, as after a possible recolouring all red vertices belong to~$L_i$, while $v$ belongs to a different component of $G-(D\cup D_1\cup \cdots \cup D_{i-1})$ than $L_i$. So, we proved the claim.
\end{claimproof}

\medskip
\noindent
Claim~\ref{c-cor} allows us to do some specific branching once we quit the blue phase for a certain main branch at level~$i$. Namely, all we have to do is to consider all options to colour at least one vertex of $D_i$ red. We call these additional branches {\it side branches}. We distinguish between the following two cases:

\medskip
\noindent
{\bf Case 1.} $|D_i|\leq 2d+1$.\\
We consider each of the at most $2^{2d+1}$ options to colour the vertices of $D_i$ either red or blue, such that at least one vertex of $D_i$ is coloured red. Next, for each  vertex $u\in D_i$, we consider all $O(n^d)$ options to colour at most $d$ of its uncoloured neighbours blue if $u$ is red, or red if $u$ is blue. Note that the total number of side branches is $O(2^{2d+1}n^{d(2d+1)})$. As $D_i$ dominates $L_i$ and the only uncoloured vertices were in $L_i$, we obtained a red-blue colouring~$c$ of $G$. We check in polynomial time if $c$ is a red-blue $d$-colouring of $G$. If so, we stop and return~$c$. If none of the side branches yields a red-blue $d$-colouring of $G$, then by Claim~\ref{c-cor} we can safely discard the main branch under consideration.

\medskip
\noindent
{\bf Case 2.} $|D_i|\geq 2d+2$.\\
Recall that $D_i$ either induces a cycle on five vertices or is a clique.
As $|D_i|\geq 2d+2\geq 6$, we find that $D_i$ is a clique. As $|D_i|\geq 2d+2$, this means that $D_i$ must be monochromatic, and thus every vertex of $D_i$ must be coloured red. We check in polynomial time if $D_i$ contains a vertex with more than $d$ neighbours in $D$ (which are all coloured blue), or if $D$ contains a vertex with more than $d$ neighbours in $D_i$ (which are all coloured red). If so, we may safely discard the main branch under consideration due to Claim~\ref{c-cor}. 

From now on, assume that every vertex in $D_i$ has at most $d$ neighbours in~$D$, and vice versa. By construction, every uncoloured vertex belongs to $L_i-D_i$. Hence, if $V(L_i)=D_i$, we have obtained a red-blue colouring~$c$ of $G$. We check, in polynomial time, if $c$ is also a red-blue $d$-colouring of $G$. If so, we stop and return $c$. Otherwise, we may safely discard the main branch due to Claim~\ref{c-cor}. 

Now assume $V(L_i)\supsetneq D_i$. We colour all vertices in $L_i-D_i$ that are adjacent to at least $d+1$ vertices in $D$ blue; we have no choice as the vertices in~$D$ are all coloured blue. If there are no uncoloured vertices left, we check in polynomial time if the obtained red-blue colouring is a red-blue $d$-colouring of $G$. If so, we stop and return it; else we may safely discard the main branch due to Claim~\ref{c-cor}. 

Assume that we still have uncoloured vertices left. We recall that these vertices belong to $L_i-D_i$, and that by construction they have at most $d$ neighbours in~$D$.
Consider an uncoloured vertex $w_1$. As $D_i$ dominates $L_i$, we find that $w_1$ has a neighbour~$x_1$ in $D_i$. As $D$ dominates $G$, we find that $x_1$ is adjacent to some vertex $y_1\in D$. We consider all $O(n^{2d})$ possible ways to colour the uncoloured neighbours of $x_1$ and $y_1$, such that $x_1$ (which is red)  has at most $d$ blue neighbours, and $y_1$ (which is blue) has at most $d$ red neighbours.
If afterwards there is still an uncoloured vertex $w_2$, then we repeat this process: we choose a neighbour $x_2$ of $w_2$ in $D_i$
(so $x_2$ is coloured red).
We now branch again by colouring the uncoloured neighbours of $x_2$. If we find an uncoloured vertex $w_3$, then we find a neighbour $x_3$ of $w_3$ in $D_i$ 
and so on.  So, we repeat this process until there are no more uncoloured vertices. This gives us $2p$ distinct vertices $w_1,\ldots,w_p$, $x_1,\ldots, x_p$
for some integer~$p\geq 1$,
together with vertex $y_1$.The total number of side branches for the main branch is $O(n^{2pd})$.

We claim that $p\leq d$. For a contradiction, assume that $p\geq d+1\geq 3$. Let $2\leq j\leq p$. As $D_i$ is a clique that contains $x_1$ and $x_j$, we find that $x_1$ and $x_j$ are adjacent. Hence, $G$ contains the $4$-vertex path $w_jx_jx_1y_1$. By construction, $w_j$ was uncoloured after colouring the neighbours of $x_1$ and $y_1$, so $w_j$ is neither adjacent to $x_1$ nor to $y_1$. This means that $w_jx_jx_1y_1$ is an induced $P_4$ if and only if $x_j$ is not adjacent to $y_1$. Recall that $w_j$ and every vertex of $D_i$, so including $x_1$ and $x_j$, has at most $d$ neighbours in $D$. As $|D|\geq 3d+1$, this means that $D$ contains a vertex $z$ that is not adjacent to any of $w_j,x_j,x_1$. As $D$ is a clique, $z$ is adjacent to $y_1$. Hence, $x_j$ must be adjacent to $y_1$, as otherwise $w_jx_jx_1y_1z$ is an induced $P_5$; see also Figure~\ref{f-p5}. The latter is not possible as $G$ is $P_5$-free. We now find that $y_1$ is adjacent to $x_1,\ldots,x_p$, which all belong to $D_i$. As we assumed that $p\geq d+1$ and every vertex of $D$, including $y_1$, is adjacent to at most $d$ vertices of $D_i$, this yields a contradiction. We conclude that $p\leq d$.

As $p\leq d$, the total number of side branches is $O(n^{2d^2})$. Each side branch yields a red-blue colouring of $G$. We check in polynomial time if it is a red-blue $d$-colouring of $G$. If so, then we stop and return it; else we can safely discard the side branch, and eventually the associated main branch, due to Claim~\ref{c-cor}.

\medskip
\noindent
The correctness of our algorithm follows from its description. 
We now analyze its running time.
We started the algorithm with searching for a set~$D$.
 If $D$ has size at most~$3d$, then we applied  Lemma~\ref{l-smalldom}, which takes polynomial time. Otherwise we argue as follows.
The total number of main branches is $O(n^d)$. For each main branch we have either $O(2^{2d+1}n^{d(2d+1)})$ side branches (Case~1) or 
 $O(n^{2d^2})$ side branches (Case~2). 
 Hence, the total number of branches is $O(n^d2^{2d+1}n^{d(2d+1)})$.
 As $d$ is fixed, this number is polynomial. 
 As processing each branch takes polynomial time (including the construction of the $D_i$-sets and $L_i$-graphs), we conclude that our algorithm runs in polynomial time. This completes the proof.
\qed
\end{proof} 

\noindent
We now consider the case $H=P_3+P_4$.

\begin{theorem}\label{t-p4p3}
For every $d\geq 2$, the \dcut{} problem is polynomial-time solvable for $(P_3+P_4)$-free graphs.
\end{theorem}

\begin{proof}
Let $d\geq 2$ and recall that throughout the proof we assume that $d$ is a constant.
Let $G$ be a connected $(P_3+P_4)$-free graph on $n$ vertices. We may assume that $G$ contains at least one induced $P_4$, else we apply Theorem~\ref{t-p5}.
Throughout our proof we make repeatedly use of the following claim.

\begin{nclaim}\label{c-constant}
{\it Let $P$ be an induced $P_4$ of $G$. Every connected component $F$ in $G-N[V(P)]$ is a complete graph and has only a constant number (namely at most $2^{2d}$) of distinct red-blue colourings.}
\end{nclaim}

\begin{claimproof}
As $P$ is an induced $P_4$, we find that $G-N[V(P)]$ is $P_3$-free. Hence, every connected component $F$ in $G-N[V(P)]$ is a complete graph. 
As $F$ is a complete graph, $F$ must be monochromatic in any red-blue $d$-colouring of $G$ if $|V(F)| \geq 2d+1$. Hence, $F$ has at most $2^{2d}$ possible red-blue colourings, which is a constant number, as $d$ is constant.
\end{claimproof}

\medskip
\noindent
We now divide the proof into two cases. First, we decide in polynomial time whether there exists a red-blue $d$-colouring of $G$ in which some induced $P_4$ is monochromatic. If we do not find such a red-blue $d$-colouring, then we search in polynomial time for a red-blue $d$-colouring in which every induced $P_4$ is bichromatic. By Observation~\ref{o-cut-colouring} this proves the theorem.
	
	\medskip
	\noindent	
	{\bf Step 1.} Decide if $G$  has a red-blue $d$-colouring, for which some induced $P_4$ of $G$ is monochromatic.
	
\medskip
\noindent
Our algorithm performs Step~1 as follows.	
We branch over all induced $P_4$s of~$G$. As there are at most $O(n^4)$ induced $P_4$s in $G$, this yields $O(n^4)$ branches. For each branch, we do as follows.
Let $P$ be the current induced $P_4$ of $G$, which we assume is monochromatic.
We assume without loss of generality that every vertex of $P$ is blue. 
Let $$N_1= N(V(P))\; \mbox{and}\; N_2= V(G) \setminus N[V(P)].$$ 
By Claim~\ref{c-constant}, every connected component of $G[N_2]$ is a complete graph. 
By definition, every vertex will have at most $d$ neighbours of the other colour in every red-blue $d$-colouring of $G$.
We branch over all $O(n^{4d})$ possible ways to colour the vertices of~$N_1$. 

For each branch we do as follows.
First suppose that  $V(P) \cup N_1$ is monochromatic. As the vertices of $P$ are all blue, this means that all vertices of $N_1$ are blue. We observe that $G$ has a red-blue $d$-colouring in which all vertices of $V(P)\cup N_1$ are blue if and only if there is a connected component $F$ of $G[N_2]$ such that $G[V(P) \cup N_1 \cup V(F)]$ has a red-blue $d$-colouring. The reason for this is that if we find such a connected component $F$ of $G[N_2]$, then we can safely colour the vertices of all other connected components of $G[N_2]$ blue.
For each connected component $F$ of $G[N_2]$ we do as follows.
By Claim~\ref{c-constant}, the number of distinct red-blue $d$-colourings of $F$ is constant. We branch by trying each such colouring. If some colouring of $V(F)$ yields a red-blue $d$-colouring of $G[V(P) \cup N_1 \cup V(F)]$,
then we are done. If after considering all
$O(n)$
 connected components of $G[N_2]$  we have not found a red-blue $d$-colouring, then we discard this branch. 

\medskip
\noindent	
Now assume that $V(P) \cup N_1$~is not monochromatic in our current branch. This means that the set of red vertices in $N_1$, which we denote by $S$, is nonempty.
Observe that $S$ contains at most $4d$ vertices, as $S\subseteq N(V(P))$ and every vertex of $P$ is blue.
We now branch over all $O(n^{4d^2})$ possible colourings of $N(S)\setminus V(P)$.

For each branch we do as follows. We colour-process the pair of current sets of red and blue vertices. Observe that connected components of $G[N_2]$ may now contain red vertices, but if they do, then they contain at least one red vertex that is adjacent to a vertex in $S$, i.e., to a red vertex in  $N_1$.  We now make a crucial update:

\medskip
\noindent
{\it We move every new blue vertex from $N_2$ to $N_1$, while we let every new red vertex stay in $N_2$ (so the only red vertices in $N_1$ are still those that belong to $S$).}

\medskip
\noindent
We call this update an {\it $N_1$-update}. Note that after an $N_1$-update, the new connected components of $G[N_2]$ are still complete graphs, but now they consist of only red vertices and uncoloured vertices.
Moreover, every uncoloured vertex belongs to $N_2$ and only has blue neighbours in $N_1$ (as the red vertices in $N_1$ belong to $S$ and all vertices of $N(S)$ are coloured).

For every connected component $F$ of $G[N_2]$ that consists of uncoloured vertices only we do as follows. 
Recall that uncoloured vertices in $N_2$ only have blue neighbours in~$N_1$. Thus, we may safely colour all vertices of $F$ blue and apply an $N_1$-update. So, after this procedure, every connected component~$F$ of $G[N_2]$ with an uncoloured vertex contains a red vertex. Recall that at least one red vertex of $F$ must have a red neighbour in~$S\subseteq N_1$,

For every connected component~$F$ of $G[N_2]$ that has an uncoloured vertex and that has size at least $2d+1$, we do as follows. As $F$ is a complete graph on at least $2d+1$ vertices, $F$ is monochromatic in every red-blue $d$-colouring of~$G$.  From the above, we know that $F$ also contains a red vertex. Hence, we must colour every uncoloured vertex of $F$ red. Afterwards, all connected components of $G[N_2]$ with an uncoloured vertex have size at most $2d$.

For every connected component $F$ of $G[N_2]$ with an uncoloured vertex we also check if all the uncoloured vertices of $F$ only have red or uncoloured neighbours (note that the latter belong to $F$ as well). If so, then we safely colour all uncoloured vertices of $F$ red. If afterwards there are no uncoloured vertices anymore, we check if the resulting colouring is a red-blue $d$-colouring of $G$. If so, we are done, and otherwise we discard the branch. 
Suppose there are still uncoloured vertices left. Then these uncoloured vertices must belong to $N_2$, and we denote the connected components of $G[N_2]$ that contain at least one uncoloured vertex by $F_1,\ldots, F_q$ for some $q\geq 1$. In summary, we have proven:

\begin{nclaim}\label{c-cor3}
{\it Every $F_i$ is a \emph{complete} graph of size at most $2d$ that\\[-20pt]
\begin{itemize}
\item [(i)] contains a red vertex with a red neighbour in $S\subseteq N_1$;
\item [(ii)] contains an uncoloured vertex with a blue neighbour in $N_1$; 
\item [(iii)] does not contain any blue vertices; and
\item [(iv)] does not contain any uncoloured vertices that have a red neighbour in $N_1$.
\end{itemize}}
\end{nclaim}

\noindent
By Claim~\ref{c-cor3}, there exists at least one red vertex in $N_1$ that has a red neighbour in some connected component $F_i$. We now distinguish between two cases.

\medskip
\noindent
{\bf Case~1.} There exists a red vertex $r \in N_1$ that has a red neighbour in two distinct connected components $F_i$ and $F_j$, say in $F_1$ and $F_2$.

\medskip
\noindent
Let $r_1 \in V(F_1)$ be a red neighbour of $r$ in $F_1$, and let $u\in V(F_1)$ be an uncoloured vertex in $F_1$. 
Similarly, let $r_2 \in V(F_2)$ be a red neighbour of $r$ in $F_2$.  
By Claim~\ref{c-cor3}, we have that $F_1$ is a complete graph, so $r_1$ and $u$ are adjacent, and thus $G$ contains the $4$-vertex path $Q=u_1r_1rr_2$.

\begin{figure}[t]
	\centering

\begin{tikzpicture}
\begin{scope}
\node[bvertex](p1) at (1,4){};
\node[bvertex](p2) at (2,4){};
\node[bvertex](p3) at (3,4){};
\node[bvertex](p4) at (4,4){};

\draw[edge](p1) -- (p2);
\draw[edge](p2) -- (p3);
\draw[edge](p3) -- (p4);

\draw[nicered, thick] (0,2.5) rectangle (2,3.5);
\draw[lightblue, thick] (2.5,2.5) rectangle (5,3.5);
\node[] (n1) at (-0.3, 3){$N_1$};
\node[] (n2) at (-0.3, 1){$N_2$};
\node[] (p4) at (-0.3,4) {$P$};

\draw[] (0,0)			 rectangle	(1,2);
\draw[] (1 + 1/3,0) 	rectangle	(2 + 1/3,2);
\draw[] (2 + 2/3,0)	 rectangle	(3 + 2/3,2);
\draw[] (4,0) 			rectangle	(5,2);

\node[] (f1) at (0.5, -0.3){$F_1$};
\node[] (f1) at (1.5 + 1/3, -0.3){$F_2$};
\node[] (f1) at (2.5 + 2/3, -0.3){$F_3$};
\node[] (f1) at (4.5, -0.3){$F_4$};

\node[rvertex, label = above:$r$] (r) at (1 + 1/6, 2.8){};
\node[rvertex, label= above left:$r_1$](r1) at (0.5, 1.5){};
\node[rvertex, label= above right:$r_2$](r2) at (1.5+ 1/3, 1.5){};
\node[evertex, label = below:$u$] (u) at (0.5,0.5){};

\draw[tedge] (u) -- (r1);
\draw[tedge] (r1) -- (r);
\draw[tedge] (r) -- (r2);

\node[bvertex] (b) at (3+ 5/6, 2.8){};
\node[evertex] (b1) at (2.5 + 2/3, 1.5){};
\node[evertex] (b2) at (4.5, 1.5){};

\draw[tedge](b1) -- (b);
\draw[tedge](b) -- (b2);
\end{scope}

\end{tikzpicture}
	\caption{Illustration of Case 1. As $Q=u_1r_1rr_2$ is induced, the displayed $P_3+P_4$ cannot be induced. White vertices indicate uncoloured vertices.}\label{f-p3p4-case1}
	\end{figure}

As $r$ is red, Claim~\ref{c-cor3} tells us that $u$ is not adjacent to $r$.
As $u$ and $r_2$ are in different connected components of $G[N_2]$, we also have that $u$ is not adjacent to $r_2$. For the same reason, $r_1$ and $r_2$ are not adjacent.
Hence, $Q$ is in fact an induced~$P_4$ of $G$. 

As $G$ is $(P_3+P_4)$-free and $Q$ is an induced $P_4$, any induced $P_3$ in $G$ contains a vertex that is adjacent to at least one vertex of $Q$.  In particular, this implies that every blue vertex in $N_1$ that has some uncoloured neighbour in at least two connected components of $\{F_3,\ldots,F_q\}$ (if $q\geq 4$)
has a neighbour in $Q$. The reason is that these two uncoloured neighbours cannot be adjacent to red vertices in $N_1$ (by Claim~\ref{c-cor3}) or to vertices in connected components of $G[N_2]$ to which they do not belong.
See Figure~\ref{f-p3p4-case1} 
for an illustration.

Since the vertices of $Q$ are either red or uncoloured, they may have at most $d$ blue neighbours (as else they would have been coloured blue due to the colour-processing). Hence, we find that there are at most $4d$  blue vertices in $N_1$ with uncoloured neighbours in at least two connected components of $\{F_3,\ldots,F_q\}$. We denote this particular set of blue vertices in $N_1$ by $T$, so we have that $|T| \leq 4d$.
	
By Claim~\ref{c-cor3}, every $F_i$ has at most $2d$ vertices (so a constant number). Hence, $|V(F_1)|\leq 2d$ and $|V(F_2)|\leq 2d$. 
Let $U$ be the set of uncoloured vertices in $V(F_1)\cup V(F_2)$.
We branch over all $O(2^{4d}n^{4d^2})$ 
colourings of $G[N(T) \cup U]$. 

For each branch we do as follows.
We first observe that every uncoloured vertex $x$ must belong to $V(F_i)$ for some $ i \in \{3, \dots, q\}$. 
The neighbours of $x$ in~$N_1$ cannot belong to $S\cup T$, as we already coloured all vertices of $N(S\cup T)$. Hence, any neighbour $y$ of $x$ in $N_1$ must be in $N_1\setminus (S\cup T)$, so in particular $y\notin S$ and thus $y$ must be blue.
Moreover, as $y$ does not belong to $T$ either,  any other uncoloured neighbours of $y$ must all belong to $F_i$ as well (this follows from the definition of $T$).

Let $W$ consist of all vertices coloured so far.
From the above, we conclude that it is enough to decide if we can extend  for each $i\in \{3,\ldots,q\}$, the current red-blue colouring of $G[W]$ to a red-blue $d$-colouring of $G[W\cup V(F_i)]$. For each $i\in \{3,\ldots,q\}$, this leads to another $2^{2d}$ branches, as $F_i$ has size at most $2d$ by Claim~\ref{c-cor3}.
So, as we can consider the colourings of every $F_i$ ($i\geq 3$) independently, the total number of new branches is $2^{2d}q=O(n)$.
If for all $i\in \{3,\ldots,q\}$ we can find an extended red-blue $d$-colouring of $G[W\cup V(F_i)]$, then we have obtained a red-blue $d$-colouring of $G$. 
Otherwise we discard the branch.

\medskip
\noindent
{\bf Case~2.} Every red vertex in $N_1$ has red neighbours in at most one connected component of $\{F_1,\ldots,F_q\}$.

\medskip
\noindent
By definition, every $F_i$ has at least one uncoloured vertex. By Claim~\ref{c-cor3}, every $F_i$ also has at least one red vertex with a red neighbour in $N_1$.
Consider~$F_1$. Let $x$ be an uncoloured vertex of $F_1$, and let $r_1$ be a red vertex of $F_1$ with a red neighbour $r$ in $N_1$, so $r\in S$.
Recall that $F_1$ is a complete graph by Claim~\ref{c-cor3}. Hence, $x$ and $r_1$ are adjacent, and thus $xr_1r$ is a $P_3$.
However, $x$ and $r$ are non-adjacent, since $r \in S$ and the neighbourhood of $S$ is coloured. 
This means that  $J$ is even an induced~$P_3$ in $G$.

By definition of a red-blue $d$-colouring and due to the colour-processing, at most $3d$ blue vertices in $G$ can have a neighbour on $J$. Let $\Gamma$ be the set of these blue vertices, so $|\Gamma|\leq 3d$. 
We now branch by considering all colourings of the uncoloured vertices of $N(\Gamma)\cup V(F_1)$. As $|\Gamma|\leq 3d$ and
$F_1$ has at most $2d$ vertices by Claim~\ref{c-cor3}, the number of colourings to consider is at most $O(2^{2d}n^{3d^2})$.
	 
For each branch we do as follows. We first colour-process the pair of sets of red vertices and blue vertices.
We then perform an $N_1$-update, such that afterwards every $F_i$ again only contains red and uncoloured vertices. Again we colour the uncoloured vertices of some $F_i$ red if none of them has a blue neighbour in $N_1$. 
Moreover, we recall that if some $F_i$ consists of only uncoloured vertices, then -- as  uncoloured vertices in $N_2$ only have blue neighbours in~$N_1$ -- we may safely colour all vertices of $F_i$ blue and apply an $N_1$-update. So, again it holds that every $F_i$ with an uncoloured vertex contains a red vertex.

Assume we still have at least one uncoloured vertex left (otherwise we are done). As we coloured the vertices of $F_1$, we find that there exists some $F_i$ with $i\geq 2$, say $F_2$, that contains at least one uncoloured vertex. As we did not colour all vertices of $F_2$ red, we find that $F_2$ contains an uncoloured vertex $u$ with a blue neighbour~$b$ in $N_1$.
We note that $b$ does not belong to $\Gamma$, as we coloured all vertices of~$N(\Gamma)$.

\begin{figure}[t]
	\centering
\begin{tikzpicture}

\begin{scope}[shift = {(0,0)}]
\node[bvertex](p1) at (1,4){};
\node[bvertex](p2) at (2,4){};
\node[bvertex](p3) at (3,4){};
\node[bvertex](p4) at (4,4){};

\draw[edge](p1) -- (p2);
\draw[edge](p2) -- (p3);
\draw[edge](p3) -- (p4);

\draw[nicered, thick] (0,2.5) rectangle (2,3.5);
\draw[lightblue, thick] (2.5,2.5) rectangle (5,3.5);
\node[] (n1) at (-0.3, 3){$N_1$};
\node[] (n2) at (-0.3, 1){$N_2$};
\node[] (p4) at (-0.3,4) {$P$};

\draw[] (0,0)			 rectangle	(1.5,2);
\draw[] (1.75,0) 	rectangle	(3.25,2);
\draw[] (3.5,0)	 rectangle	(5,2);

\node[] (f1) at (0.75, -0.3){$F_1$};
\node[] (f1) at (2.5 , -0.3){$F_2$};
\node[] (f1) at (4.25, -0.3){$F_3$};

\node[rvertex, label = above :$r$] (r) at (1 + 1/6, 2.8){};
\node[rvertex, label=  left:$r_1$](r1) at (0.6, 1.5){};
\node[evertex, label = below:$x$] (x) at (0.6,0.5){};

\draw[tedge] (x) -- (r1);
\draw[tedge] (r1) -- (r);
 
\node[bvertex, label = above:$b$] (b) at (3.7, 2.8){};
\node[evertex, label= below:$u$] (b1) at (2.8, 1.5){};
\node[evertex, label = below: $v$] (b2) at (4.25, 1.5){};
\node[rvertex, label= below:$r_2$] (r2) at (2.1,1.5){};

\draw[tedge](b1) -- (b);
\draw[tedge](b) -- (b2);
\draw[tedge](b1) -- (r2);
\end{scope}

\end{tikzpicture}
	\caption{Illustration of Case 2. As $J = rr_1u$ is induced, the displayed $P_3+P_4$ cannot be induced. White vertices indicate uncoloured vertices.}\label{f-p3p4-case2}
	\end{figure}

First suppose that $b$ has an uncoloured neighbour $v$ in some $F_i$ with $i\geq 3$, say in~$F_3$, and also that $b$ also has a non-neighbour $r_2$ 
(that is either red or uncoloured) 
in $F_2$ or $F_3$, say in~$F_2$. As $F_2$ is a complete graph, $u$ and $r_2$ are adjacent. Hence, $r_2ubv$ is an induced $P_4$ of $G$.

We note that $b$ is not adjacent to any vertex of $\{r,r_1,x\}$, as $b$ does not belong to $\Gamma$. Note also that $u$ and $v$ are not adjacent to $r_1$ and $x$, as $r_1$ and $x$ are in a different connected component of $G[N_2]$ than $u$ and $v$. Moreover, as $r$ is red, Claim~\ref{c-cor3} tells us that $u$ and $v$ are not adjacent to $r$ either. 

Finally, we consider~$r_2$. We note that $r_2$ is not adjacent to~$r$. Namely, if $r_2$ is red, then this holds by the Case~2 assumption. If $r_2$ is uncoloured, then this holds due to Claim~\ref{c-cor3}. Moreover,  $r_2$ is not adjacent to $r_1$ and $x$, as $r_1$ and~$x$ belong to a different connected component of $G[N_2]$ than $r_2$.
 However, we now find an induced $P_3+P_4$ of $G$ consisting of  $\{r,r_1,x\}$ and $\{r_2,u,b,v\}$, see Figure~\ref{f-p3p4-case2}. As $G$ is $(P_3+P_4)$-free, this is not possible. 
 
 From the above we conclude that the following holds for every blue vertex $b''\in N_1$ that has an uncoloured neighbour:
 
 \begin{itemize}
 \item [(i)]   vertex $b'$ has only uncoloured neighbours in at most one $F_i$, or\\[-10pt]
 \item [(ii)]  vertex $b'$ is adjacent to every 
 vertex of 
 every~$F_i$, in which $b'$ has an uncoloured neighbour.
 \end{itemize}
 
 \noindent
 If case (i) holds, then we say that $b$ is of {\it type}~(i). If case (ii) holds, then we say that $b$ is of {\it type}~(ii). 
 
Recall that any uncoloured vertex, which must belong to $N_2$, only has blue neighbours in $N_1$ by Claim~\ref{c-cor3}. 
We say that a connected component $F_i$  with an uncoloured vertex is an {\it individual} component of $G[N_2]$ if the uncoloured vertices of $F_i$ have only blue neighbours of type~(i) in $N_1$; else we say that $F_i$ is a {\it collective} component of $G[N_2]$.

Suppose $b$ is a blue vertex of $N_1$ that has an uncoloured neighbour and that is of type~(ii). As $b$ is blue, $b$ is has at most $d-1$ red neighbours, else its remaining neighbours would have been coloured blue by the colour-processing, and hence, $b$ would not have an uncoloured neighbour. As $b$ is of type~(ii), it follows that $b$ is adjacent to every vertex of every $F_i$ that contains an uncoloured neighbour of~$b$. Recall that every $F_i$ that has an uncoloured vertex also contains a red vertex. Hence, we find that $b$ has uncoloured neighbours in at most $d-1$ collective components.  As $d\geq 2$, the same statement also holds if $b$ is of type~(i). Hence, we have proven the following claim:
 
 \begin{nclaim}\label{c-cor4}
{\it Every blue vertex in $N_1$ has uncoloured neighbours in at most $d-1$ collective components.}
\end{nclaim}

\begin{figure}[t]
	\centering
\begin{tikzpicture}

\begin{scope}[shift = {(6.3,0)}]
\node[bvertex](p1) at (1,4){};
\node[bvertex](p2) at (2,4){};
\node[bvertex](p3) at (3,4){};
\node[bvertex](p4) at (4,4){};

\draw[edge](p1) -- (p2);
\draw[edge](p2) -- (p3);
\draw[edge](p3) -- (p4);

\draw[nicered, thick] (0,2.5) rectangle (2,3.5);
\draw[lightblue, thick] (2.5,2.5) rectangle (5,3.5);
\node[] (n1) at (-0.3, 3){$N_1$};
\node[] (n2) at (-0.3, 1){$N_2$};
\node[] (p4) at (-0.3,4) {$P$};

\draw[] (0,0)			 rectangle	(1,2);
\draw[] (1 + 1/3,0) 	rectangle	(2 + 1/3,2);
\draw[] (2 + 2/3,0)	 rectangle	(3 + 2/3,2);
\draw[] (4,0) 			rectangle	(5,2);

\node[] (f1) at (0.5, -0.3){$F_1$};
\node[] (f1) at (1.5 + 1/3, -0.3){$F_h$};
\node[] (f1) at (2.5 + 2/3, -0.3){$F_i$};
\node[] (f1) at (4.5, -0.3){$F_j$};

\node[rvertex, label = above :$r$] (r) at (1 + 1/6, 2.8){};
\node[rvertex, label=  left:$r_1$](r1) at (0.6, 1.5){};
\node[evertex, label = below:$x$] (x) at (0.6,0.5){};

\draw[tedge] (x) -- (r1);
\draw[tedge] (r1) -- (r);
 
\node[bvertex, label = above:$b_1$] (b1) at (3, 2.8){};
\node[bvertex, label = above:$b_2$] (b2) at (4.3, 2.8){};
\node[evertex, label= below:$u$] (u) at (1.5+1/3, 1.5){};
\node[evertex, label = below: $v$] (v) at (2.5 + 2/3, 1.5){};
\node[evertex, label= below:$w$] (w) at (4.5,1.5){};

\draw[tedge](b1) -- (u);
\draw[tedge](b1) -- (v);
\draw[tedge](b2) -- (v);
\draw[tedge](b2) -- (w);
\end{scope}

\end{tikzpicture}
	\caption{Illustration of Case 2  where we show that two vertices $b_1$ and $b_2$ of type~(ii) must have either disjoint neighbourhoods in~$N_2$ or one neighbourhood in $N_2$ must be contained in the other. White vertices indicate uncoloured vertices.}\label{f-p3p4-case22}
	\end{figure}

\noindent
Now suppose there exist two blue vertices $b_1$ and $b_2$ in $N_1$ that are both of type~(ii) and connected components $F_h$, $F_i$ and $F_j$ with $h,i,j\geq 2$ such that:

\begin{itemize}
\item $b_1$ has an uncoloured neighbour $u$ in $F_h$, whereas $b_2$ has no neighbour in $F_h$;
\item $b_1$ and $b_2$ both have an uncoloured neighbour in $F_i$, which we may assume to be the same neighbour as $b_1$ and $b_2$ are of type~(ii); we denote this common uncoloured neighbour by $v$;
\item $b_2$ has an uncoloured neighbour $w$ in $F_j$, whereas $b_1$ has no neighbour in $F_j$.
\end{itemize}

\noindent
See Figure~\ref{f-p3p4-case22} for an illustration. Recall that $b_1$ and $b_2$ do not belong to $\Gamma$, as else they would not have an uncoloured neighbour. This implies the following. 
If $b_1$ and $b_2$ are not adjacent, then $J$, together with $\{u,b_1,v,b_2\}$, induces a $P_3+P_4$ in $G$. 
If $b_1$ and $b_2$ are adjacent, then $J$, together with $\{u,b_1,b_2,w\}$, induces a $P_3+P_4$ in $G$. Both cases are not possible due to $(P_3+P_4)$-freeness of $G$. Hence, we conclude that such vertices $b_1$ and $b_2,$ and such connected components $F_h$, $F_i$, $F_j$ do not exist.

We assume that without loss of generality the connected components with an uncoloured vertex are $F_2,\ldots,F_r$ for some $r\geq 2$. 
Suppose $G[N_2]$ has collective components. Then, due to the above and Claim~\ref{c-cor4}, we can now partition the collective components of $\{F_2,\ldots,F_r\}$ in blocks $D_1,\ldots,D_p$ for some $p\geq 1$,  such that the following holds:

\begin{itemize}
\item for $i=1,\ldots, p$, there exists a blue vertex $b_i$ in $N_1$ that is adjacent to all vertices of every collective component of $D_i$;
\item for $i=1,\ldots, p$, every $D_i$ contains at most $d-1$ collective components, each of which contain at most~$2d$ vertices by Claim~\ref{c-cor3}; and
\item every blue vertex in $N_1$ is adjacent to vertices of collective components of at most one $D_i$.
\end{itemize}

\noindent
Let $W$ be the set of vertices of $G$ that have already been coloured. Let ${\cal D}=\{D_1,\ldots,D_p\}$. We add every individual component $F$ that does not belong to some block in ${\cal D}$ to ${\cal D}$ as a block $\{F\}$.
Note that by definition, every blue vertex in $N_1$ that is of type~(i) has uncoloured neighbours in only one individual component. 
Hence, due to the above, $G$ has a red-blue $d$-colouring that extends the colouring of $G[W]$ if and only if for every $D\in {\cal D}$ it holds that $G[W\cup \bigcup_{F\in D}V(F)]$ has a red-blue $d$-colouring.
Therefore, we can consider each $D\in {\cal D}$ separately as follows. We consider 
all  colourings of the uncoloured vertices in $G[V(D_i)]$ and for each of these colourings we check if we obtained a red-blue $d$-colouring of $G[W\cup \bigcup_{F\in D}V(F)]$.
As $D_i$ contains at most $d-1$ components, each with at most $2d$ vertices, the number of these colourings is at most $2^{2d(d-1)}$.
Note that there are at most $p\leq n$ number of blocks that we must consider.

If the above leads to a red-blue $d$-colouring of $G$, then we are done. Else, we discard the branch.

\medskip
\noindent
Suppose we have not found a red-blue $d$-colouring of $G$ after processing all the branches. If $G$ is still a yes-instance, then in any red-blue $d$-colouring of $G$, every induced $P_4$ of $G$ is bichromatic. Therefore, our algorithm will perform the following step:

\medskip
\noindent
{\bf Step 2.} Decide if $G$ has a red-blue $d$-colouring, for which every induced $P_4$ of $G$ is bichromatic.

\medskip
\noindent	
We first prove a structural claim. Assume $G$ has a red-blue $d$-colouring. Let $B$ be the set of blue vertices. Let $R$ be the set of red vertices. Assume that $G[B]$ has connected components $D^B_1,\ldots, D^B_p$ for some $p\geq 1$ and that $G[R]$ has connected components $D^R_1,\ldots, D^R_q$ for some $q\geq 1$. If $q\geq 2$, then we obtain another red-blue $d$-colouring of $G$ by changing the colour of every vertex in every $D^R_h$ with $h\in \{2,\ldots,r\}$ from red to blue. Hence, we may assume $q=1$. Suppose $p\geq 2$. As $G$ is connected, every $D^B_i$ contains a vertex that is adjacent to a vertex in $D^R_1$. Hence, we obtain another red-blue $d$-colouring of $G$ by changing the colour of every vertex in every $D^B_i$ with $i\in \{2,\ldots,p\}$ from blue to red. So we may also assume $p=1$. 
Recall that we have ruled out the case where $G$ has a red-blue $d$-colouring in which some induced $P_4$ of $G$ is monochromatic. 
This means that both $G[B]$ and $G[R]$  have diameter at most $2$, else we would have a monochromatic~$P_4$. So we have shown the following claim:

 \begin{nclaim}\label{c-cor5}
{\it  If $G$ has a red-blue $d$-colouring, then $G$ has a red-blue $d$-colouring in which the sets of blue and red vertices, respectively, each induce subgraphs of $G$ that have diameter at most~$2$.}
\end{nclaim}

\noindent
By Claim~\ref{c-cor5}, we will now search for a red-blue $d$-colouring of $G$ in which the sets of blue and red vertices, respectively, each induce subgraphs of $G$ that have diameter at most~$2$.
We say that such a red-blue $d$-colouring has {\it diameter at most~$2$}.

Recall that $G$ is not $P_4$-free. Hence, we can find an induced path $P$ on four vertices in $O(n^4)$ time by brute force. We now branch over all $O(n^{4d})$ colourings of $N[V(P)]$. Let $N_2=V(G) \setminus N[V(P)]$.
If $N_2 = \emptyset$, then we check if we obtained a red-blue $d$-colouring of $G$. If so, then we are done, and else we discard the branch. 

Suppose $N_2\neq \emptyset$.
As $G$ is $(P_3+P_4)$-free, every connected component of $G[N_2]$ is a complete graph.
Let $F_1, \dots, F_k$, for some integer $k \geq 1$, be the connected components of $G[N_2]$.
We colour-process the pair consisting of the current sets of blue and red vertices.

If all vertices are coloured, then we check if we obtained a red-blue $d$-colouring of $G$. If so, then we are done, and else we discard the branch.
Suppose not every vertex is coloured. This means there exists some $F_i$, say $F_1$, with an uncoloured vertex~$x$.
As we colour-processed and $x$ has not been coloured, we find that  $x$ has at most $d$ red neighbours and at most $d$ blue neighbours. Let $B_x \subseteq N_1$ be the set of blue neighbours of $x$ in $N_1$ and $R_x \subseteq N_1$ be the set of red neighbours of $x$ in $N_1$. 
By Claim~\ref{c-constant}, the number of red-blue colourings of the uncoloured vertices of $F_1$ is constant (at most $2^{2d}$). Let $U$ be the set of uncoloured vertices of $F$.
We branch over all 
$O(n^{2d^2})$ colourings of $G[N(B_x \cup R_x) \cup U]$. 

We now colour the remaining uncoloured vertices, which must all belong to $N_2\setminus V(F_1)$ as follows. If $x$ is blue, then 
 any blue vertex of $N_2\setminus V(F_1)$ must have a neighbour in $B_x$ in any red-blue $d$-colouring of $G$ that has diameter at most~$2$ and that is an extension of the current red-blue colouring.  As we already coloured all vertices in $N(B_x)$, this means that we must colour any uncoloured vertex of $N_2$ red. Similarly, if $x$ is red, then we must colour any uncoloured vertex of $N_2$ blue. We check if the resulting colouring is a red-blue $d$-colouring of $G$. If yes, we have found a solution, otherwise we discard the branch. 

 In the end we either found a red-blue $d$-colouring of $G$, or we have discarded every branch. In that case, the algorithm returns that $G$ has no red-blue $d$-colouring. This completes the description of our algorithm.
  
 \medskip
 \noindent
 The correctness of the algorithm follows from its description. The maximum number of branches in Step~1 is 
  \[O\left(n^4 \cdot n^{4d} \cdot \left( n + n^{4d^2} \cdot \left( n^{4d^2} \cdot n + n^{3d^2}\cdot n\right) \right) \right),\] 
 and in Step~2, it is $O(n^4\cdot n^{4d} \cdot n^{2d^2})$.
 So, in total we have at most
 $n^{O(d^3)}$ branches and each branch can be processed in polynomial time   
  (in particular, colour-processing takes polynomial time due to Lemma~\ref{l-process}; $N_1$-updates can be done in polynomial time; and we can also check in polynomial time if a red-blue colouring of $G$ is a red-blue $d$-colouring). Thus, the running time of our algorithm is polynomial. 
	 \qed
\end{proof}

\noindent
As our final result in this section, we prove the following.

\begin{theorem}\label{t-p1}
For every graph $H$ and every $d\geq 2$, if \dcut{} is polynomial-time solvable for $H$-free graphs, then it is so for $(H+P_1)$-free graphs.
\end{theorem}

\begin{proof}
Suppose that \dcut{} is polynomial-time solvable for $H$-free graphs. Let $G$ be a connected $(H+P_1)$-free graph. If $G$ is $H$-free, the result follows by assumption. If $G$ is not $H$-free, then $V(G)$ contains a set $U$ such that $G[U]$ is isomorphic to $H$. As $G$ is $(H+P_1)$-free, $U$ dominates $G$. As $H$ is fixed, $|U|=|V(H)|$ is a constant. Hence, in this case we can apply Lemma~\ref{l-smalldom}. \qed
\end{proof}

\section{NP-Completeness Results}\label{s-np}

In this section we show our \NP-completeness results for \dcut{} for $d\geq 2$.
As \dcut{} is readily seen to be in \NP\ for 
each
$d\geq 1$, we only show \NP-hardness in our proofs. 

We first focus on the case where $d\geq 3$. For this, we need some additional terminology.
 An edge colouring of a graph $G=(V,E)$ with colours red and blue is called a \emph{red-blue edge colouring} of $G$, which is 
a \emph{red-blue edge $d$-colouring} of $G$ if every edge of $G$ is adjacent to at most $d$ edges of the other colour and both colours are used at least once.
Now,  $G$ has a red-blue edge $d$-colouring if and only if $L(G)$ has a red-blue $d$-colouring.
A set $S\subseteq V$ is \emph{monochromatic} if all edges of $G[S]$ are coloured alike.

\begin{theorem}\label{t-line}
For every $d\geq 3$, the \dcut{} problem is \NP-complete for line graphs.
\end{theorem}

\begin{proof}
We first define the known \NP-complete problem we reduce from. Let $X= \{x_1,x_2,...,x_n\}$ be a set of logical  variables and  ${\cal C} = \{C_1, C_2, . . . , C_m\}$ be a set of clauses over $X$. The problem  {\sc Not-All-Equal Satisfiability} asks whether $(X,{\cal C})$ has a {\it satisfying not-all-equal} truth assignment $\phi$ that is, $\phi$ sets at least one literal true and at least one literal false in each $C_i$. This problem remains \NP-complete even if each clause consists of three distinct literals that are all positive~\cite{Sc78}.
Let $(X,{\mathcal C})$ be such an instance, where $X = \{ x_1, x_2, \dots, x_n \}$ and ${\mathcal C} = \{C_1, C_2, \dots, C_m\}$.
We construct, in polynomial time, a graph~$G$; see also Figure~\ref{fig:line}:

\begin{itemize}
\item Build a clique $S=\{v^S_{x_1},\ldots,v^S_{x_n}\} \cup \{v_1^{c_1},\ldots,v_{d-2}^{c_1}\}\cup \cdots \cup \{v_1^{c_m},\ldots,v_{d-2}^{c_m}\}$.\\[-7pt]
\item Build a clique $\overline{S}=\{v^{\overline{S}}_{x_1},\ldots,v^{\overline{S}}_{x_n}\} \cup \{u_1^{c_1},\ldots,u_{d-2}^{c_1}\}\cup \cdots \cup \{u_1^{c_m},\ldots,u_{d-2}^{c_m}\}$.\\[-7pt]
\item For every $x \in X$, add cliques $V_x=\{v_1^x, \dots, v_{d-1}^x\}$ and $V_{\overline{x}}=\{v_1^{\overline{x}}, \dots, v_{d-1}^{\overline{x}}\}$.\\[-7pt]
\item For every $x \in X$, add a vertex $v_x$ with edges $v_xv_x^S$, $v_xv_x^{\overline{S}}$, $v_xv_1^x, \dots, v_xv_{d-1}^x$, $v_xv_1^{\overline{x}},\ldots, v_xv_{d-1}^{\overline{x}}$.\\[-7pt]
\item For every $C\in {\cal C}$, add a clause vertex $v_c$ with edges $v_cv_1^c, \dots,  v_c v_{d-2}^c$, $v_c u_1^c, \dots, \linebreak v_c u_{d-2}^c$,
 and if $C=\{x_i,x_j,x_k\}$, also
 add a vertex  $v_c^{x_i}$ to $V_{x_i}$, a vertex $v_c^{x_j}$ to $V_{x_j}$ and a vertex $v_c^{x_k}$ to $V_{x_k}$, and add the edges $v_c v_c^{x_i}$, $v_c v_c^{x_j}$, $v_c v_c^{x_k}$.\\[-7pt]
\item Add, if needed, some auxiliary vertices to $S, \overline{S}, V_{x_1}, \ldots, V_{x_n}, 
V_{\overline{x_1}}, \dots, V_{\overline{x_n}}$, such that in the end all these sets are cliques of size at least~$2d+2$.
\end{itemize}

\begin{figure}[t]
\centering
\includegraphics[width = \textwidth]{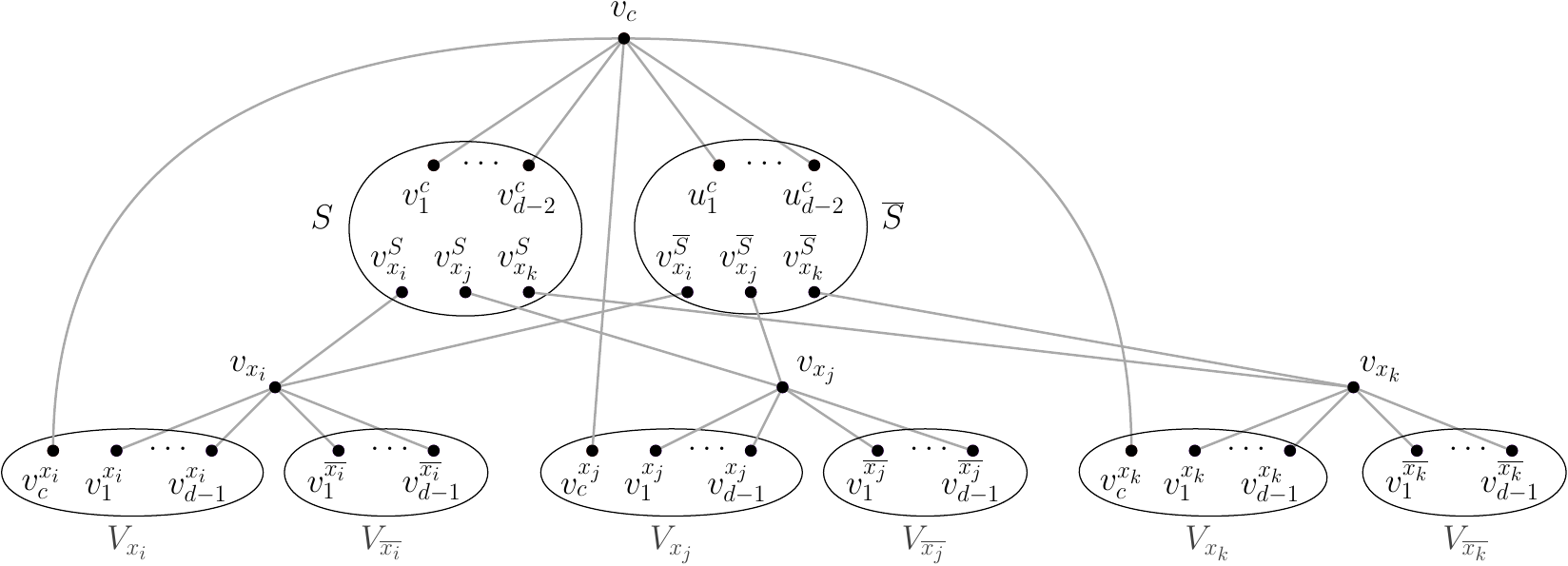}
\caption{An example of vertices in the reduction related to clause $C = \{x_i,x_j,x_k\}$.}\label{fig:line}
\end{figure}

\noindent
We claim that $(X,{\mathcal C})$ has a satisfying not-all-equal truth assignment if and only if
the line graph
 $L(G)$ has a $d$-cut.
 Recall that, by Observation~\ref{o-cut-colouring}, $L(G)$ has a $d$-cut if and only if $L(G)$ has a red-blue $d$-colouring. 
 Furthermore,
 $L(G)$ has a red-blue $d$-colouring if and only if $G$ has a red-blue edge $d$-colouring.
 Hence, we will show that $(X,{\mathcal C})$ has a satisfying not-all-equal truth assignment if and only if $G$ has a red-blue edge $d$-colouring.
 
 First suppose $(X,{\cal C})$ has a satisfying not all-equal truth assignment.
We colour all edges in $S$ red and in $\overline{S}$ blue.
For every $x \in X$ set to true, we colour the edges in  $V_x$ red and those in $V_{\overline{x}}$ blue. 
For every $x \in X$ set to false, we colour the edges in $V_x$ blue and those in $V_{\overline{x}}$ red. 
Consider an edge $uv$, with $v \in \{v_x, v_c\; |\; x \in X, c\in C\}$. Then $u$ is contained in a clique $D \in \{S, \overline{S}, V_{x_1}, V_{\overline{x_1}}, \dots, V_{x_n}, V_{\overline{x_n}}\}$.
Colour $uv$ with the same colour as the edges of~$D$.

Now, let $D \in \{S, \overline{S}, V_{x_1}, V_{\overline{x_1}}, \dots, V_{x_n}, V_{\overline{x_n}}\}$.
Every $uu' \in E(D)$ is adjacent to only edges of the same colour.
For $u \in V(D)$ and $v \in \{v_x, v_c\; |\; x \in X, c\in C\}$, the edge~$uv$ has the same colour as all edges in $D$. Since $S$ and $\overline{S}$ have different colours and $V_x$ and $V_{\overline{x}}$ have different colours for every $x \in X$, $uv$ has at most $d$~adjacent edges of each colour.
Hence, we obtained a red-blue edge $d$-colouring of $G$.
 
\medskip
\noindent
 Now suppose that $G$ has a red-blue edge $d$-colouring. 
We prove a series of claims:

\begin{nclaim}\label{cl:line-unicolor}
Every clique $D \in \{S, \overline{S}, V_{x_1}, V_{\overline{x_1}}, \dots, V_{x_n}, V_{\overline{x_n}}\}$ is monochromatic. 
\end{nclaim}

\begin{claimproof}
First assume $G[D]$ has a red edge $uv$ and a blue edge $uw$.
As $|D| \geq 2d+2$, we know that $u$ is incident to at least $2d+1$ edges. Hence, we may assume without loss of generality that $u$ is incident to at least $d+1$ red edges. However, now the blue edge $uw$ is adjacent to $d+1$ red edges, a contradiction.
As every $u\in D$ is incident to only edges of the same colour and $D$ is a clique, it follows that $D$ is monochromatic.
\end{claimproof}

\medskip
\noindent
By Claim~\ref{cl:line-unicolor}, we can speak about the {\it colour} (either red or blue) of a clique $D$ if $D$ belongs to $\{S, \overline{S}, V_{x_1}, V_{\overline{x_1}}, \dots, V_{x_n}, V_{\overline{x_n}}\}$.

\begin{nclaim}\label{cl:line-edge}
For each $ x \in X$ and $c \in C$, each edge from $v_x$ or $v_c$ to a vertex~in a clique $D \in \{S, \overline{S}, V_{x_1}, V_{\overline{x_1}}, \dots, V_{x_n}, V_{\overline{x_n}}\}$ has the same colour as~$D$.
\end{nclaim}

\begin{claimproof}
This follows directly from the fact that $|D|\geq 2d+1$.
\end{claimproof}

\begin{nclaim}\label{cl:line-bicolor}
The cliques $S$ and $\overline{S}$ have different colours if and only if  for every variable $x \in X$, it holds that $V_x$ and $V_{\overline{x}}$ have different colours.
\end{nclaim}

\begin{claimproof}
First suppose $S$ and $\overline{S}$ have different colours, say $S$ is red and $\overline{S}$ is blue. For a contradiction, assume there exists a variable $x \in X$, such that $V_x$ and $V_{\overline{x}}$ have the same colour, say blue. By Claim~\ref{cl:line-edge}, we have that the $2d-2$ edges between $v_x$ and $V_x\cup V_{\overline{x}}$ and the edge $v_xv^{\overline{S}}_x$ are all blue, while $v_xv^S_x$ is red. Hence, the red edge $v_xv^S_x$ is adjacent to at least $2d-1\geq d+1$ blue edges, a contradiction.  

Now suppose that for all $x \in X$, $V_x$ and $V_{\overline{x}}$ have different colours, say $V_x$ is red and $V_{\overline{x}}$ is blue.
 For a contradiction, assume that $S$ and $\overline{S}$ have the same colour, say blue. Let $x\in X$.
 By Claim~\ref{cl:line-edge}, we have that the edges between $v_x$ and $V_x$ are red, while all other edges incident to $v_x$ are blue. Now every (red)
 edge between $v_x$ and $V_x$ is incident to $d+1$ blue edges, a contradiction.
\end{claimproof}

\begin{nclaim}\label{cl:line-bicolorS}
The cliques $S$ and $\overline{S}$ have different colours.
\end{nclaim}

\begin{claimproof}
For a contradiction, assume $S$ and $\overline{S}$ have the same colour, say blue. By Claim~\ref{cl:line-edge}, we have that $v_xv^S_x$ is blue.
By Claim~\ref{cl:line-bicolor}, we find that for every $x\in X$, $V_x$ and $V_{\overline{x}}$ have the same colour. If $V_x$ and $V_{\overline{x}}$ are both red,
then the $2d-2$ edges between $v_x$ and $V_x\cup V_{\overline{x}}$ are red due to  Claim~\ref{cl:line-edge}. Consequently, the blue edge $v_xv^S_x$ is adjacent to $2d-2\geq d+1$ red edges. This is not possible. Hence,  $V_x$ and $V_{\overline{x}}$ are blue, and by Claim~\ref{cl:line-edge}, all edges between $v_x$ and $V_x\cup V_{\overline{x}}$ are blue as well. This means that every edge of $G$ is blue, a contradiction.
\end{claimproof}

\begin{nclaim}\label{cl:line-vx}
For every clause $C= \{x_i,x_j,x_k\}$ in ${\cal C}$, the cliques $V_{x_i}$, $V_{x_j}$ and $V_{x_k}$ do not all have the same colour. 
\end{nclaim}

\begin{claimproof}
For a contradiction, assume $V_{x_i}$, $V_{x_j}$ and $V_{x_k}$ have the same colour, say blue.
By Claim~\ref{cl:line-bicolorS}, $S$ and $\overline{S}$ are  coloured differently, say $S$ is red and $\overline{S}$ is blue.
 By Claim~\ref{cl:line-edge}, we have that the three edges $v_cv_c^{x_i}$, $v_cv_c^{x_j}$ and $v_cv_c^{x_k}$ are all blue, just like the $d-2$ edges between $v_c$ and $\overline{S}$, while every edge between $v_c$ and $S$ is red. Consider such a red edge~$e$. We find that $e$ is adjacent to $d+1$ blue edges,
a contradiction.
\end{claimproof}

\medskip
\noindent
For each variable $x$, if the clique $V_x$ is coloured red, then set~$x$ to true, and else to false.
By Claim~\ref{cl:line-vx}, this yields a satisfying not-all-equal truth assignment.
\qed
\end{proof}

\noindent
We now show that the case $H=3P_2$ is hard. The gadget in our \NP-hardness reduction is neither $2P_4$-free nor $P_6$-free nor $P_7$-free.

\begin{theorem}\label{t-3p2new}
For every $d\geq 2$, the \dcut{} problem is \NP-complete for $3P_2$-free graphs of radius~$2$ and diameter~$3$.
\end{theorem}

\begin{proof}
We first define the known \NP-complete problem we reduce from.
Let $X= \{x_1,x_2,\cdots,x_n\}$ be a set of variables. Let ${\cal C} = \{C_1, C_2, \cdots, C_m\}$ be a set of clauses over $X$. The {\sc $3$-Satisfiability} problem asks whether $(X,{\cal C})$ has a
{\it satisfying} truth assignment~$\phi$, that is, $\phi$ sets at least one literal true in each $C_i$.
Darmann and D\"ocker~\cite{DD21} proved that {\sc $3$-Satisfiability} is \NP-complete even for instances in which:
\begin{enumerate}
\item each variable occurs as a positive literal in exactly two clauses and as a negative literal in exactly two other clauses, and
\item each clause consists of three distinct literals that are either all positive or all negative.
\end{enumerate}
Let $X= \{x_1,x_2,\dots,x_n\}$ for some $n\geq 1$ and ${\cal C} = \{C_1, \dots ,C_p,D_1,\ldots,  D_q\}$ where each $C_j$ consists of three distinct positive literals, and each $D_j$ consists of three distinct negative literals. We may assume without loss of generality that $p\geq 4$ and $q\geq 4$, as otherwise the problem is trivial to solve by using brute force.

From $(X,{\cal C})$, we construct a graph $G=(V,E)$ as follows. We introduce two vertices $C$ and $D$ and let the other vertices of $G$ represent either variables or clauses. That is,
we introduce a clique $K=\{C_1,\ldots,C_p,C\}$; a clique $K'=\{D_1,\ldots,D_q,D\}$ and an independent set $I=\{x_1,\ldots,x_n\}$, such that $K$, $K'$, $I$ are pairwise disjoint and $V=K\cup K'\cup I$. For every $h\in \{1,\ldots,n\}$ and every $i\in \{1,\ldots,p\}$, we add an edge between $x_h$ and $C_i$ if and only if $x_h$ occurs as a literal in $C_i$. 
For every $h\in \{1,\ldots,n\}$ and every $j\in \{1,\ldots,q\}$, we add an edge between $x_h$ and $D_j$ if and only if $x_h$ occurs as a literal in $D_j$. We also add the edge $CD$.
See Figure~\ref{f-3p2} for an example. 

\begin{figure}
\centering
\begin{tikzpicture}

\node[rvertex, label = left: $C$](C) at (3.5, 4){};
\node[bvertex, label = right: $D$] (D) at (6.5, 4){};
\draw[edge] (C) -- (D);

\foreach \i  in {1,...,4}{
	\node[rvertex, label = above: $C_\i$](C\i) at (\i, 2.5){};
}

\begin{scope}[shift = {(5,0)}]
\foreach \i  in {1,...,4}{
	\node[bvertex, label = above: $D_\i$](D\i) at (\i, 2.5){};
}\end{scope}

\node[rvertex, label = below: $x_1$](x1) at (2.5, 0){};
\foreach \i  in {2,...,5}{
	\node[bvertex, label = below: $x_\i$](x\i) at (\i+1.5, 0){};
}
\node[rvertex, label = below: $x_6$](x6) at (7.5, 0){};

\draw[edge](C1) -- (x1);
\draw[edge](C1) -- (x2);
\draw[edge](C1) -- (x3);

\draw[edge](C2) -- (x1);
\draw[edge](C2) -- (x3);
\draw[edge](C2) -- (x4);

\draw[edge](C3) -- (x2);
\draw[edge](C3) -- (x5);
\draw[edge](C3) -- (x6);

\draw[edge](C4) -- (x4);
\draw[edge](C4) -- (x5);
\draw[edge](C4) -- (x6);

\draw[edge](D1) -- (x1);
\draw[edge](D1) -- (x2);
\draw[edge](D1) -- (x4);

\draw[edge](D2) -- (x1);
\draw[edge](D2) -- (x3);
\draw[edge](D2) -- (x5);

\draw[edge](D3) -- (x2);
\draw[edge](D3) -- (x4);
\draw[edge](D3) -- (x6);

\draw[edge](D4) -- (x3);
\draw[edge](D4) -- (x5);
\draw[edge](D4) -- (x6);

\draw[dashed] (2,-0.5) rectangle (8,0.3);
\draw[dashed] (0.5,2.2) rectangle (4.5,4.3);
\draw[dashed] (5.5,2.2) rectangle (9.5,4.3);
\node[](k) at ( 0,3.25) {$K$};
\node[](k) at ( 10,3.25) {$K'$};
\node[](k) at ( 1.5,0) {$I$};

\end{tikzpicture}
\caption{The graph $G$ for 
$X = 
\{x_1, \dots, x_6\}$ and $\mathcal{C} = \{ \{ x_1, x_2, x_3\},\{x_1, x_3, x_4 \},$ $\{x_2, x_5, x_6 \},\{x_4, x_5, x_6 \}, 
 \{\overline{x_1}, \overline{x_2},  \overline{x_4}\},\{\overline{x_1}, \overline{x_3},  \overline{x_5}\},  \{\overline{x_2}, \overline{x_4},  \overline{x_6}\},  \{\overline{x_3}, \overline{x_5},  \overline{x_6}\}\}$.
  For readability the edges inside the cliques $K$ and $K'$ are not shown. }\label{f-3p2}
\end{figure}

As every edge must have at least one end-vertex in $K$ or $K'$, and $K$ and $K'$ are cliques, we find that $G$ is $3P_2$-free. Moreover, $G$ has radius~$2$, as the distance from $C$ or $D$ to any other vertex in $G$ is at most~$2$. In addition, the distance from a vertex in $I\cup (K\setminus \{C\}) \cup (K'\setminus \{D\})$ to any other vertex in $G$ is at most~$3$. Hence, $G$ has diameter at most~$3$.

We claim that $(X,{\cal C})$ has a satisfying truth assignment if and only if $G$ has a $2$-cut.

First suppose that $(X,{\cal C})$ has a satisfying truth assignment~$\phi$.
In $I$, we colour for 
$h\in \{1,\ldots,n\}$, vertex $x_h$ red if $\phi$ sets $x_h$ to be true and blue if $\phi$ sets $x_h$ to be false. We colour all the vertices in $K$ red and the vertices in $K'$ blue. 

Consider a vertex $x_h$ in $I$. First suppose that $x_h$ is coloured red. 
As each literal appears in exactly two clauses from $\{D_1,\ldots, D_q\}$, we find that $x_h$ has only two blue neighbours (which all belong to $K'$). 
Now suppose that $x_h$ is coloured blue. As each literal appears in exactly two clauses from $\{C_1,\ldots, C_p\}$, we find that $x_h$ has only two red neighbours (which all belong to $K$). 
Now consider a vertex~$C_i$ in $K$, which is coloured red. As $C_i$ consists of three distinct positive literals and $\phi$ sets at least one positive literal of $C_i$ to be true, we find that $C_i$ is adjacent to at most two blue vertices in $I$.
Hence, every $C_i$ is adjacent to at most two blue vertices.
Now consider a vertex~$D_j$ in $K'$, which is coloured blue. As $D_j$ consists of three distinct negative literals and $\phi$ sets at least one negative literal of $D_j$ to be
 true,
 we find that $D_j$ is adjacent to at most two red vertices in $I$. Hence, every $D_j$ is adjacent to at most two red vertices. Finally, we note that $C$, which is coloured red, is adjacent to exactly one blue neighbour, namely $D$, while $D$ has only one red neighbour, namely $C$.
The above means that we obtained a red-blue $2$-colouring of $G$. By Observation~\ref{o-cut-colouring}, this means that $G$ has a $2$-cut.

Now suppose $G$ has a $2$-cut. By Observation~\ref{o-cut-colouring}, this means that $G$ has a red-blue $2$-colouring $c$. As $|K|=p+1\geq 5$ and $|K'|=q+1\geq 5$,  both $K$ and $K'$ are monochromatic. Say $c$ colours every vertex of $K$ red. For a contradiction, assume $c$ colours every vertex of $K'$ red as well. As $c$ must colour at least one vertex of $G$ blue, 
this means that $I$ contains a blue vertex $x_h$.  
As each variable occurs as a positive literal in exactly two clauses and as a negative literal in exactly two other clauses,
we now find that a blue vertex, $x_i$, has two red neighbours in $K$ and two red neighbours in $K'$, so four red neighbours in total, a contradiction. We conclude that $c$ must colour every vertex of $K'$ blue.

Recall that every $C_i$ and every $D_j$ consists of three literals. Hence, every vertex in $K\cup K'$ has three neighbours in $I$. As every vertex $C_i$ in $K$ is red, this means that at least one neighbour of $C_i$ in $I$ must be red.
As every vertex $D_j$ in $K'$ is blue, this means that at least one neighbour of $D_j$ in $I$ must be blue.
Hence, setting $x_i$ to true if $x_i$ is red in $G$ and to false if $x_i$ is blue in~$G$ gives us the desired truth assignment for $X$. 

\medskip
\noindent
Now, we consider the case where $d\geq 3$. We adjust $G$ as follows. We first modify $K$ into a larger clique by adding for each $x_h$, a set $L_h$ of $d-3$ new vertices. We also modify $K'$ into a larger clique by adding for each $x_h$, a set $L'_h$ of $d-3$ vertices. For each $h\in \{1,\ldots,n\}$ we make $x_h$ complete to both $L_h$ and to $L_h'$.
Finally, we add additional edges between vertices in $K\cup L_1\cup \ldots \cup L_n$ and vertices in $K'\cup L_1'\cup \ldots \cup L_n'$, in such a way that every vertex in $K\setminus \{C\}$ has $d-2$ neighbours in $K'\setminus \{D\}$, and vice versa. 
The modified graph $G$ is still $3P_2$-free, has radius~$2$ and diameter~$3$, and also still has size polynomial in $m$ and $n$.
The remainder of the proof uses the same arguments as before. \qed
\end{proof}

\section{Conclusions}\label{s-con}

We considered the natural generalization of \mc{} to \dcut~\cite{GS21} and proved dichotomies for graphs of bounded diameter and graphs of bounded radius.
We also started a systematic study on the complexity of \dcut{} for $H$-free graphs. 
While for $d=1$, there still exists an infinite number of non-equivalent open cases, we were able to obtain for every $d\geq 2$,
 an almost-complete complexity classification of \dcut{} for $H$-free graphs, with only three non-equivalent open cases left 
 if $d\geq 3$. 
 We finish our paper with some open problems on $H$-free graphs resulting from our systematic study.

We recall that {\sc $1$-Cut} is polynomial-time solvable for claw-free graphs~\cite{Bo09}, while we showed that \dcut{} is \NP-complete even for line graphs if $d\geq 3$. 
We also recall the recent result of Ahn et al.~\cite{AELPS25} who proved that $2$-{\sc Cut} is \NP-complete for claw-free graphs.
What is the computational complexity of $2$-{\sc Cut} for line graphs?

Finally, we recall the only three non-equivalent open cases  $H=2P_4$, $H= P_6$, $H=P_7$ for \dcut{} on $H$-free graphs for $d\geq 2$.
We aim to address these cases as future work.

\medskip
\noindent
{\bf Acknowledgments.} We thank Carl Feghali and \'Edouard Bonnet for fruitful discussions.

\bibliographystyle{splncs04}
\bibliography{ref.bib}

\end{document}